\documentclass[12pt,reqno]{amsart}
\usepackage{a4wide}

\numberwithin{equation}{section}

\allowdisplaybreaks

\usepackage{color}

\newtheorem{theorem}{Theorem}[section]
\newtheorem{proposition}[theorem]{Proposition}
\newtheorem{corollary}[theorem]{Corollary}
\newtheorem{lemma}[theorem]{Lemma}

\theoremstyle{definition}

\theoremstyle{remark}

\renewcommand{\epsilon}{\varepsilon }

\newcommand{\R}{\mathbb{R}}

\begin{document}
\title[Multiple Boundary Peak Solution for Critical Elliptic Neumann System]%
{Multiple  Boundary  Peak Solution for Critical Elliptic System with Neumann boundary }

\author{  Yuxia Guo, Shengyu Wu and Tingfeng Yuan}

\address{  Department of Mathematical  Science, Tsinghua University, Beijing 100084, P.R.China}
\email{yguo@tsinghua.edu.cn}

\address{  Department  of Mathematical Science, Tsinghua University, Beijing 100084, P.R.China}
\email{wusy21@mails.tsinghua.edu.cn}

\address{  Department  of Mathematical Science, Tsinghua University, Beijing 100084, P.R.China}
\email{ytf22@mails.tsinghua.edu.cn}





\let\thefootnote\relax\footnotetext{ Guo is supported by 2023YFA1010002, NSFC (No. 12271283)}

\date{}


\begin{abstract}
We consider the following elliptic system with Neumann boundary:
\begin{equation*}
   \begin{cases}
   -\Delta u + \mu u=v^p,\;\;\; &\hbox{in } \Omega,\\
   -\Delta v + \mu v=u^q,\;\;\; &\hbox{in } \Omega,\\
   \frac{\partial u}{\partial n} = \frac{\partial v}{\partial n} = 0, &\hbox{on } \partial\Omega,\\
   u>0,v>0, &\hbox{in } \Omega,
   \end{cases}
\end{equation*}
where $\Omega \subset \R^N$ is a smooth bounded domain, $\mu$ is a positive  constant and $(p,q)$ lies in the critical hyperbola:
$$
\dfrac{1}{p+1} + \dfrac{1}{q+1}  =\dfrac{N-2}{N}.
$$
By using the Lyapunov-Schmidt reduction technique, we establish the existence of infinitely many solutions to above system. These solutions have multiple peaks that are located on the boundary $\partial \Omega$. Our results show that the geometry of the boundary $\partial\Omega,$ especially its mean curvature, plays a crucial role on the existence and the behaviour of the solutions to the problem.
\end{abstract}

\maketitle


\section{Introduction}

In this paper, we are concerned with the following  nonlinear elliptic system with Neumann boundary condition:
\begin{equation}\label{equ-1}
\begin{cases}
-\Delta u + \mu u=v^p,\;\;\; &\hbox{in } \Omega,\\
-\Delta v + \mu v=u^q,\;\;\; &\hbox{in } \Omega,\\
\frac{\partial u}{\partial n} = \frac{\partial v}{\partial n} = 0, &\hbox{on } \partial\Omega,\\
u>0,v>0, &\hbox{in } \Omega,
\end{cases}
\end{equation}
where $\Omega$ is a smooth bounded domain in $\R^N$, $\mu > 0$ is a  constant and $(p,q)$ is a pair of positive numbers lying on the critical hyperbola:
\begin{equation}\label{cri}
 \frac{1}{p+1} +\frac{1}{q+1} =\frac{N-2}{N}.
\end{equation}
Without loss of generality, we may assume that $p\leq \frac{N+2}{N-2}\leq q.$

System (\ref{equ-1}) appears in various models of applied science. An instance of its presence can be observed in the chemotactic aggregation model introduced by Keller and Segel, see \cite{LinLiTakagi, NiTakagishapeof}. Additionally, it exhibits a close association with the Gierer-Meinhardt system, which was originally formulated to investigate the phenomenon of biological pattern formation, see \cite{NiTakagipoint, WeiJDE97} and references therein.

If $u=v$, the system (\ref{equ-1}) is reduced to a single elliptic Neumann equation:
\begin{equation}\label{reduce-equ-1}
   \begin{cases}
   -\Delta u + \mu u=u^{p},\;\;\; &\hbox{in } \Omega,\\
   \frac{\partial u}{\partial n}  = 0, &\hbox{on } \partial\Omega,\\
   u>0, &\hbox{in } \Omega,
   \end{cases}
\end{equation}
where $ p = \frac{N+2}{N-2}$.

This type of equation has been extensively studied in recent years. It is shown  that the geometry of the boundary $\partial \Omega$, especially its mean curvature, plays a crucial role on the existence and behaviour of solution to (\ref{reduce-equ-1}). For example,  in subcritical case with $p < \frac{N+2}{N-2}$, Lin, Ni and Takagi proved that equation (\ref{reduce-equ-1}) has only trivial solution for small $\mu$ while nonconstant solution exists for large $\mu$, see \cite{LinLiTakagi}. Subsequently, Ni and Takagi proved in \cite{NiTakagishapeof, NiTakagipoint} that the least energy solution attains its maximum at a single point on $\partial \Omega$ for each $\mu$. Moreover, as $\mu$ tends to $ + \infty$, these maximal points tend to a boundary point that maximizes the mean curvature $H(x)$ of the boundary $\partial \Omega$. Additionally, high energy solutions have been studied. As a result, one or multiple peak solutions are constructed. These solutions exhibits different blow-up behavior, including blow-up at one or several points on $\partial \Omega$ \cite{YanPacific, topological-nontrivial, LiYYC1stable, WeiJDE97} or in the interior of $\Omega$ \cite{Interiorpeak, multi-peak-interior-subcritical}, and blow up at multiple points both on $\partial \Omega$ and in the interior of $\Omega$ \cite{interior-boundary-sub}. For more any other related results, we refer the readers to  \cite{role-meancurvature, Geometry-boundary, topological-nontrivial, critical-concentrate, WeiJDE97} and references therein.

In the critical case $p =\frac{N+2}{N-2},$ the investigation of boundary peak solutions, including both single or multiple peaks, was explored in \cite{role-meancurvature, GuiCriticalmulti, WangZQCritical, WangZQconstruction, Yan07arbitrary-kpeak}. Moreover, the singular behavior of specific solutions were studied in \cite{Criticaloneboundarypeak} and \cite{NiPanTakagi}, where the authors focused on single boundary peak solution and the least energy solution respectively. Additionally, equation (\ref{reduce-equ-1}) admits the phenomenon of bubble accumulation, where multiple bubble concentrate at the same point on the boundary $\partial \Omega$. This  phenomenon is discussed in \cite{Bubble-accumulation}.

If we consider the equation (\ref{reduce-equ-1}) with supercritical exponent, that is $p > \frac{N+2}{N-2}$, then Lin, Ni found in \cite{LinLiconjuecture} that equation (\ref{reduce-equ-1}) only possesses trivial solution for sufficiently small $\mu$, while nonconstant solution exists for large $\mu$ if $\Omega$ is a ball, similar to the subcritical case ($p<\frac{N+2}{N-2}$) discussed in \cite{LinLiTakagi}. Based on this fact, they proposed the Lin-Ni conjecture, which states that for a smooth and bounded domain $\Omega$, and $p > 1$, equation (\ref{reduce-equ-1}) has only trivial solution for small $\mu$, whereas nontrivial solutions exist for large $\mu$. To our best knowledge, Rey and Wei are the first to  provide a negative answer to the Lin-Ni conjecture in \cite{critical-interior} for $N \geq 5$ and  $p = \frac{N+2}{N-2}.$ Subsequently, Wang, Wei, and Yan disproved this conjecture in \cite{Neumann-Linli} by considering a non-convex domain $\Omega$ that satisfies certain geometric assumptions. To be more precise, they assumed the following three conditions on the domain $\Omega$:

\hspace*{\fill}

\textbf{(H1)}: If $y = (y_1, y_2,\cdots,y_n) \in \Omega$, then $y = (y_1, y_2,\cdots, -y_i,\cdots, y_n) \in \Omega$ for $i=3,4,\cdots,N$.

\textbf{(H2)}: If $(r , 0 , y'' ) \in \Omega$, then $(rcos \theta, r sin \theta, y'') \in \Omega$ for any $\theta \in (0,2\pi)$, where $y'' \in \R^{N-2}$.

\textbf{(H3)}: Let $ T = \partial \Omega \cap \{ y_3 = y_4 =\cdots, = y_N = 0 \}$. There is a connected component $\Gamma$ of $T$ such that the mean curvature $H(x)$ on $\Gamma$ is a negative constant $\gamma$.

\hspace*{\fill}

Under these assumptions, they proved that:

\hspace*{\fill}

\textbf{Theorem A}: Suppose $N\geq 3$, $\Omega$ is a smooth bounded domain which satisfies (H1), (H2) and (H3) and $\mu>0$ is a fixed constant. Then equations (\ref{reduce-equ-1}) has infinitely many positive solutions whose energy can be arbitrarily large.

\hspace*{\fill}

Theorem  A implies falsehood of Lin-Ni's conjecture for some specified domain $\Omega$. We would like to point out that, in a subsequent work \cite{Lin-Liconvex}, Wang, Wei and Yan extended their findings by removing the assumption of nonconvexity on $\Omega$ and disproved Lin-Li's conjecture for $N\geq 4.$ For more results, we refer the interested readers to \cite{LinLisuper, DelPino-LinNi, Druet, Yadava1, Yadava2, Yadava3}. These references provide further insights and discussions on the topic.

The aim of this paper is to generalize the results obtained in Theorem A to system \eqref{equ-1}. Among any other things, we obtain the existence of infinitely many boundary  peak solution.  We would like to point out that, due to the lack of the compactness of embedding and the trading off between $u$ and $v$, it accounts for a much rather structure and a rather different characterization of the critical growth. It is not impossible to apply critical point theory directly to obtain the existence of multiple solutions.

Before the statement of the main result, we introduce some notations. Firstly, we make the following transformation: $u(y) \mapsto \varepsilon^{-\frac{N}{q+1}} u \left(  \frac{y}{\varepsilon} \right)$ and $v(y) \mapsto \varepsilon^{-\frac{N}{p+1}} v \left(  \frac{y}{\varepsilon} \right)$, then \eqref{equ-1} becomes:
\begin{equation}\label{equ-2}
\begin{cases}
-\Delta u + \mu \varepsilon^2 u=v^p,\;\;\; &\hbox{in } \Omega_\varepsilon,\\
-\Delta v + \mu \varepsilon^2 v=u^q,\;\;\; &\hbox{in } \Omega_\varepsilon,\\
\frac{\partial u}{\partial n} = \frac{\partial v}{\partial n} = 0, &\hbox{on } \partial\Omega_\varepsilon,\\
u>0,v>0, &\hbox{in } \Omega_\varepsilon,
\end{cases}
\end{equation}
where $\Omega_{\epsilon} = \{ y| \epsilon y \in \Omega  \}$. Note that formally, if we let $\epsilon \rightarrow 0$, then we get the  "limit" problem corresponding to  (\ref{equ-1}), namely the following system in $\mathbb{R}^N$:
\begin{equation}\label{4}
\begin{cases}
-\Delta U =|V|^{p-1}V,\;\;\; \hbox{in } \mathbb R^N,\\
-\Delta V =|U|^{q-1}U,\;\;\; \hbox{in } \mathbb R^N,\\
(U,V)\in \dot{W}^{2,\frac{p+1}{p}}(\mathbb R^N) \times \dot{W}^{2,\frac{q+1}{q}}(\mathbb R^N),
\end{cases}
\end{equation}
where $N\geq 3$ and $(p,q)$ satisfy \eqref{cri}. It is shown in \cite{ref1} that (\ref{4}) poccesses a positive ground state $(U,V)$. By Sobolev embeddings, there holds that

\begin{equation}\label{5}
\begin{cases}
\dot{W}^{2,\frac{p+1}{p}}(\mathbb R^N) \hookrightarrow \dot{W}^{1,p^*}(\mathbb R^N)  \hookrightarrow L^{q+1}(\mathbb R^N),\\
\dot{W}^{2,\frac{q+1}{q}}(\mathbb R^N) \hookrightarrow \dot{W}^{1,q^*}(\mathbb R^N)  \hookrightarrow L^{p+1}(\mathbb R^N),
\end{cases}
\end{equation}
with
$$
\frac{1}{p^*}=\frac{p}{p+1}-\frac{1}{N}=\frac{1}{q+1}+\frac{1}{N},\;\; \frac{1}{q^*}=\frac{q}{q+1}-\frac{1}{N}=\frac{1}{p+1}+\frac{1}{N},
$$
and so the following energy functional is well-defined in $ \dot{W}^{2,\frac{p+1}{p}}(\mathbb R^N) \times \dot{W}^{2,\frac{q+1}{q}}(\mathbb R^N) $:
$$
I_0(u,v) := \int_{\mathbb R^N} \nabla u \cdot \nabla v -\frac{1}{p+1} \int_{\mathbb R^N} |v|^{p+1} -\frac{1}{q+1} \int_{\mathbb R^N} |u|^{q+1}.
$$
According to \cite{ref2}, the ground state is radially symmetric and decreasing up to a suitable translation. Thanks to \cite{ref3} and \cite{ref4}, the positive ground state $(U_{1,0},V_{1,0})$ of \eqref{4} is unique with $U_{1,0}(0)=1$ and the family of functions
\begin{equation}\label{6}
(U_{\lambda,x}(y),V_{\lambda,x}(y)) = (\lambda^{\frac{N}{q+1}} U_{1,0} (\lambda(y-x)), \; \lambda^{\frac{N}{p+1}} V_{1,0}(\lambda(y-x))    )
\end{equation}
for any $\lambda >0, x\in \mathbb R^N$ also solves system \eqref{4}. We should point out that the sharp asymptotic behavior of the ground states to \eqref{4} (see \cite{ref3}) and the non-degeneracy for \eqref{4} at each ground state (see \cite{ref5}) play an important role to construct bubbling solutions especially using Lyapunov-Schmidt reduction methods.

On the other hand, We would like to point out that from condition (H2), the component $\Gamma$ in conditin (H3) is a circle. Without loss of generality, we can assume $\Gamma = \{ (y_1, y_2, 0): y_1^2 + y_2^2 =1 \}$, where $0$ is the zero vector in $\mathbb R^{N-2}$. Define
\[
\begin{split}
H_s = \{ & (u,v) : u, v\in H^1(\Omega_\varepsilon),\; u,v \hbox{ is even in } y_h,\; h = 2,\cdots, N, \\
& u(r\cos \theta , r\sin \theta, y'') = u\left(r\cos \left( \theta + \frac{2\pi j}{k}  \right), r\sin \left( \theta + \frac{2\pi j}{k}  \right) , y''  \right), \\
& v(r\cos \theta , r\sin \theta, y'') = v\left(r\cos \left( \theta + \frac{2\pi j}{k}  \right), r\sin \left( \theta + \frac{2\pi j}{k}  \right) , y''  \right), \\
& \hbox{for } j=1,\cdots, k-1 \},
\end{split}
\]
and
$$
x_j = \left( \frac{1}{\varepsilon} \cos \frac{2(j-1)\pi}{k}, \frac{1}{\varepsilon} \sin \frac{2(j-1)\pi}{k}, 0 \right), j = 1,\cdots, k,
$$
where $0$ is the zero vector in $\mathbb R^{N-2}$ and we let $\epsilon = k^{-\frac{N-2}{N-3}}$. Then it follows imediately that $x_j \in \partial \Omega_{\varepsilon}$.

Let $(PU_{\Lambda, x_j}, PV_{\Lambda, x_j})$  be the unique solution of

\begin{equation}\label{equ-3}
\begin{cases}
-\Delta u + \mu \varepsilon^2 u=V_{\frac{1}{\Lambda},x_j}^p,\;\;\; &\hbox{in } \Omega_\varepsilon,\\
-\Delta v + \mu \varepsilon^2 v=U_{\frac{1}{\Lambda},x_j}^q,\;\;\; &\hbox{in } \Omega_\varepsilon,\\
\frac{\partial u}{\partial n} = \frac{\partial v}{\partial n} = 0, &\hbox{on } \partial\Omega_\varepsilon.
\end{cases}
\end{equation}
Here, we assume that $\Lambda \in (\delta, \delta^{-1})$, where $\delta > 0$ is a small constant. Set
$$
(PU,PV)=\left( \sum\limits_{j=1}^k PU_{\Lambda, x_j}, \sum\limits_{j=1}^k PV_{\Lambda, x_j} \right).
$$

In order to use $(PU,PV)$ as the approximate solution to problem (1.1), we need to further restrict the index $p$. Suppose $N\geq 5$ and $p$ satisfies the following condition

\hspace*{\fill}

\textbf{(A)}:If $N=5$, then $p \in (2,\frac{7}{3}]$. If $N\geq 6$, then $p \in (\frac{N+\tau}{N-2}, \frac{N+2}{N-2}]$, where $\tau = \frac{N-3}{N-2}$.

\hspace*{\fill}

Our main results of the paper is:

\begin{theorem}\label{t1}
Suppose $N \geq 5$ and $p$ satisfies condition (A). Besides, $\Omega$ is a smooth bounded domain satisfying (H1), (H2), (H3), and $\mu$ is a fixed positive constant. Then there exists $k_0>0$, such that for any $k > k_0$, problem \eqref{equ-2} admits a solution with the following form:
\[
   (u,v) = (PU + \omega_1 , PV + \omega_2)
\]
where $\omega_1$ and $\omega_2$ are error term that satisfies Neumann condition.
\end{theorem}

As a consequence, we have

\begin{corollary}
   Suppose $N \geq 5$ and $p$ satisfies condition (A). Besides, $\Omega$ is a smooth bounded domain satisfying (H1), (H2), (H3), and $\mu$ is a fixed positive constant. Then system \eqref{equ-1} has infinitely many solutions whose energy can be arbitrarily large.
\end{corollary}

Before proceeding further, we would like mention the main difficulty of this problem. Because of the Neumann boundary condition, the blow up points are naturally assumed to be on the boundary, ledding to the involvment of geometric property, especially the mean curvarure of the boundary, to the argument in energy expansion. Besides, the system is strongly indefinite and $u, v$ are trad off. Therefore, it is very complicated in each steps including the energy expansion and the estimates of the error terms, some new technique ideas are needed. We believe our results provide a use tool for other related studies concerning on system of Halmitonian type.

The paper is organized as follows. In Section 2, we proceed a finite reduction arguments.  Section 3 is devoted to the energy expansion of the approximate solution $(PU,PV).$ The proof of the Theorem \ref{t1} is given in Section 4. We put some essential and crucial estimates in Appendix.

\section{Finite Dimensional Reduction}
In this section, we shall employ finite-dimensional reduction technique to transform the original problem into a finite-dimensional counterpart. For this purpose, we introduce the following norms defined by:

$$
\|u\|_{*,1}=\sup\limits_{y\in\Omega_\varepsilon}\Big(\sum\limits_{j=1}^{k}\frac{1}{(1+
|y-x_{j}|)^{\frac{N}{q+1}+\tau}}\Big)^{-1}|u(y)|,
$$

$$
\|v\|_{*,2}=\sup\limits_{y\in\Omega_\varepsilon}\Big(\sum\limits_{j=1}^{k}\frac{1}{(1+
|y-x_{j}|)^{\frac{N}{p+1}+\tau}}\Big)^{-1}|v(y)|,
$$
and
$$
\|u\|_{**,1}=\sup\limits_{y\in\Omega_\varepsilon}\Big(\sum\limits_{j=1}^{k}\frac{1}{(1+
|y-x_{j}|)^{\frac{N}{q+1}+2+\tau}}\Big)^{-1}|u(y)|,
$$

$$
\|v\|_{**,2}=\sup\limits_{y\in\Omega_\varepsilon}\Big(\sum\limits_{j=1}^{k}\frac{1}{(1+
|y-x_{j}|)^{\frac{N}{p+1}+2+\tau}}\Big)^{-1}|v(y)|,
$$
where we choose $\tau = \frac{N-3}{N-2}$ in accordance with (A). Define
\begin{equation*}
\|(u,v)\|_* = \|u\|_{*,1}+\|v\|_{*,2}, \quad \|(u,v)\|_{**} = \|u\|_{**,1}+\|v\|_{**,2},
\end{equation*}

and
$$
Y_{i,1} = \dfrac{\partial PU_{\Lambda, x_i}}{\partial \Lambda}, \; Z_{i,1} = -\Delta Y_{i,1} + \mu \epsilon^2 Y_{i,1} = pV_{\frac{1}{\Lambda},x_i}^{p-1} \dfrac{\partial V_{\frac{1}{\Lambda}, x_i}}{\partial \Lambda},
$$

$$
Y_{i,2} = \dfrac{\partial PV_{\Lambda, x_i}}{\partial \Lambda}, \; Z_{i,2} = -\Delta Y_{i,2} + \mu \epsilon^2 Y_{i,2} = qU_{\frac{1}{\Lambda},x_i}^{q-1} \dfrac{\partial U_{\frac{1}{\Lambda}, x_i}}{\partial \Lambda},
$$
for $i=1,2,\cdots,N.$

Then, we define the following functional spaces that we work with, namely:

\[
   X= \{ (u,v) \in H^{1}(\Omega_{\epsilon}) \times H^{1}(\Omega_{\epsilon}), \dfrac{\partial u}{\partial n} =\dfrac{\partial v}{\partial n} = 0 \; \; \hbox{on}\;\; \partial \Omega_{\epsilon}, ||(u,v)||_{*}<+\infty \},
\]
and
\[
   E = \{ (u,v) \in X, \langle  \sum\limits_{i=1}^{k} Z_{i,1}, u \rangle + \langle \sum\limits_{i=1}^{k} Z_{i,2}, v \rangle = 0 \}.
\]

Since $X$ is a closed subspace of Hilbert space $H^{1}(\Omega_\epsilon) \times H^{1}(\Omega_\epsilon)$, there is an inner product defined on $X$ as:
\[
   \langle (u,v), (w,z) \rangle_E = \langle u,w \rangle_{H^{1}(\Omega_\epsilon)} + \langle v,z \rangle_{H^{1}(\Omega_\epsilon)},
\]
where $\langle u,w \rangle_{H^{1}(\Omega_\epsilon)} = \int_{\Omega_{\epsilon}} (\nabla u \nabla w + \mu \epsilon^2 uw )$. This inner product will mainly be used in the proof of Lemma \ref{existence}.

Note that our goal is to seek a solution of the form $(PU+\omega_1, PV +\omega_2 )$ for \eqref{equ-2}. By direct computation, we can verify that $(\omega_1,\omega_2)$ satisfies the following equations:
\begin{equation}\label{op-1}
L(\omega_1,\omega_2) =l +N(\omega_1,\omega_2),
\end{equation}
where
\begin{equation}\label{op-2}
\begin{split}
L(\omega_1,\omega_2) =& \Big(L_1 (\omega_1,\omega_2), L_2 (\omega_1,\omega_2)\Big) \\
=&\Big(   -\Delta \omega_1 +\mu \varepsilon^2 \omega_1 -p(PV)^{p-1} \omega_2,  -\Delta \omega_2 + \mu \varepsilon^2 \omega_2-q(PU)^{q-1} \omega_1\Big),
\end{split}
\end{equation}

\begin{equation}\label{op-3}
l= \Big(l_1,l_2\Big)= \Big( (PV)^p -\sum\limits_{j=1}^k V_{\frac{1}{\Lambda},x_j}^p,  (PU)^q -\sum\limits_{j=1}^k U_{\frac{1}{\Lambda},x_j}^q   \Big),
\end{equation}
and
\begin{equation}\label{op-4}
N(\omega_1,\omega_2)= \Big(N_1(\omega_2),N_2(\omega_1)\Big) ,
\end{equation}
with
\[
\begin{split}
N_1(\omega_2) = & (PV+\omega_2)^p -(PV)^p-p(PV)^{p-1}\omega_2  , \\
N_2(\omega_1) = & (PU+\omega_1)^q -(PU)^q-q(PU)^{q-1}\omega_1  .
\end{split}
\]

Next we consider the following problem:
\begin{equation}\label{equ-reduction}
   \begin{cases}
      L(\omega_1,\omega_2) = h + c (\sum\limits_{i=1}^{k} Z_{i,1}, \sum\limits_{i=1}^{k} Z_{i,2})\;\;\; \hbox{in} \;\;\; \Omega_{\epsilon},\\
      \frac{\partial \omega_{1}}{\partial n}=\frac{\partial \omega_{2}}{\partial n}=0 \;\;\; \hbox{on} \;\;\; \partial\Omega_{\epsilon}, \\
      \omega_1, \omega_2 \in H_s ,\\
      \langle ( \sum\limits_{i=1}^{k} Z_{i,1}, \sum\limits_{i=1}^{k} Z_{i,2} ) , ( \omega_1 ,\omega_2  ) \rangle = 0.
   \end{cases}
\end{equation}
We remark that  by the symmetry assumption on $\Omega$ and the fact that $\omega_{1}, \omega_{2} \in H_{s}$, we do not require the translational derivatives of $PU_{\Lambda,x_i}$ and $PV_{\Lambda,x_i}$.

To perform the reduction process, we first give the following  a priori estimate for solutions of equation (\ref{equ-reduction}).
\begin{lemma}\label{prioriestimate}
   Assume $N\geq 5$, $p$ satisfies condition (A) and $\omega_k=(\omega_{k,1},\omega_{k,2})$ solves \eqref{equ-reduction} for $h_k = (h_{k,1}, h_{k,2})$. If $||h_k||_{**}$ goes to zero as k goes to infinity, then $||\omega_{k}||_{*}$ also goes to zero as $k$ goes to infinity.
\end{lemma}

\begin{proof}
   We argue by contradiction. Suppose that $\omega_{k} = (\omega_{k,1}, \omega_{k,2})$ solves equation \eqref{equ-reduction} with $h=h_k$. And as $k \rightarrow \infty$, $||h||_{**}=||h_k||_{**} \rightarrow 0, \Lambda_{k} \in [\delta, \delta^{-1}]$, $||\omega_{k}||_{*} \geq c > 0$. Without loss of generality, we may assume $||\omega_{k}||_{*} = 1 $. For simplicity, we drop the subscript $k$ and write $\omega=(\omega_1,\omega_2)$, $h=(h_1,h_2)$.

 According to (\ref{op-2}) and (\ref{equ-reduction}), we have
 \begin{equation}\label{omega1}
   -\Delta w_1 + \mu\epsilon^2 \omega_1 - p(PV)^{p-1}\omega_2 = h_1 + c\sum\limits_{i=1}^{k}Z_{i,1}.
 \end{equation}

Then, based on Lemma \ref{LemmaA-3}, we deduce that
\begin{equation}\label{esti-omega1}
|\omega_1(y)| \leq C \int_{\Omega_{\epsilon}} \dfrac{1}{|y-z|^{N-2}} (  |(PV)^{p-1}(z)\omega_{2}(z)| + |h_{1}(z)| + c|\sum\limits_{i=1}^{k}Z_{i,1}| ) dz.
\end{equation}

We estimate the right three terms of (\ref{esti-omega1}) respectively.

By Lemma \ref{first-term}, we have
\[
\begin{split}
   & \quad \int_{\Omega_{\epsilon}} \dfrac{1}{|y-z|^{N-2}} |(PV)^{p-1}(z)\omega_{2}(z)| dz \\
   & \leq C||w_2||_{*,2}\left( \sum\limits_{i=1}^{k} \dfrac{1}{(1+|y-x_i|)^{\frac{N}{q+1}+\tau+ \theta}} + o(1)\sum\limits_{i=1}^{k} \dfrac{1}{(1+|y-x_i|)^{\frac{N}{q+1}+\tau }} \right).
\end{split}
\]

Moreover, it follows from Lemma B.3 and the fact $\frac{N}{q+1}+2+\tau < N$ that
\[
   \begin{split}
      & \quad \int_{\Omega_{\epsilon}} \dfrac{1}{|z-y|^{N-2}} |h_1(z)| dz \\
      & \leq ||h_1||_{**,1} \int_{\R^N} \dfrac{1}{|z-y|^{N-2}} \sum\limits_{i=1}^{k}\dfrac{1}{(1+|z-x_i|)^{\frac{N}{q+1}+2+\tau}}dz \\
      &\leq ||h_1||_{**,1} \sum\limits_{i=1}^{k}\dfrac{1}{(1+|y-x_i|)^{\frac{N}{q+1}+\tau}}.
   \end{split}
\]

Now we estimate the last term on the right side of (\ref{esti-omega1}). By Lemma \ref{LemmaA-3} and Lemma \ref{LemmaA-4}, we have
$$
|Z_{i,1}(z)| = p\left|V_{\frac{1}{\Lambda}, x_{i}}^{p-1} \dfrac{\partial V_{\frac{1}{\Lambda},x_{i}}}{ \partial \Lambda}\right| \leq C\dfrac{1}{(1+|z-x_i|)^{(N-2)p}},
$$
where $C$ depends on $N$, $p$ and $\delta$. Thus, it follows that
\[
   \begin{split}
      &\quad \int_{\Omega_{\epsilon}} \dfrac{1}{|z-y|^{N-2}} \sum\limits_{i=1}^{k}|Z_{i,1}(z)| dz \\
      &\leq C \int_{\Omega_{\epsilon}} \dfrac{1}{|z-y|^{N-2}} \sum\limits_{i=1}^{k}  \dfrac{1}{(1+|z-x_i|)^{(N-2)p}} dz\\
      & \leq C \sum\limits_{i=1}^{k} \dfrac{1}{(1+|y-x_i|)^{N-2}} \leq C \sum\limits_{i=1}^{k} \dfrac{1}{(1+|y-x_i|)^{\frac{N}{q+1}+\tau}}.
   \end{split}
\]

Similarly, we estimate $\omega_{2}$. Since $\omega_{2}$ satisfies
\begin{equation}\label{omega2}
   -\Delta w_2 + \mu\epsilon^2 \omega_2 - q(PU)^{q-1}\omega_1 = h_2 + c\sum\limits_{i=1}^{k}Z_{i,2}.
 \end{equation}
By using the similar arguments as in the estimates of $\omega_1,$ we have
$$
|\omega_2(y)| \leq C \int_{\Omega_{\epsilon}} \dfrac{1}{|y-z|^{N-2}} (  |(PU)^{q-1}(z)\omega_{1}(z)| + |h_{2}(z)| + c|\sum\limits_{i=1}^{k}Z_{i,2}| ) dz,
$$
and furthermore,
\[
   \begin{split}
     &\quad  \int_{\Omega_{\epsilon}} \dfrac{1}{|y-z|^{N-2}} |(PU)^{q-1}(z)\omega_{1}(z)|  \\
      &\leq C||w_1||_{*,1}\left( \sum\limits_{i=1}^{k} \dfrac{1}{(1+|y-x_i|)^{\frac{N}{p+1}+\tau+ \theta}}  + o(1)\sum\limits_{i=1}^{k} \dfrac{1}{(1+|y-x_i|)^{\frac{N}{p+1}+\tau }} \right) ,
   \end{split}
\]

$$
\int_{\Omega_{\epsilon}} \dfrac{1}{|y-z|^{N-2}} |h_{2}(z)| \leq C||h_2||_{**,2}\sum\limits_{i=1}^{k}\dfrac{1}{(1+|y-x_i|)^{\frac{N}{p+1} + \tau}} ,
$$
and
$$
\int_{\Omega_{\epsilon}} \dfrac{1}{|y-z|^{N-2}} |\sum\limits_{i=1}^{k}Z_{i,2}|  \leq C \sum\limits_{i=1}^{k}\dfrac{1}{(1+|y-x_i|)^{\frac{N}{p+1} + \tau}} .
$$
Next, we estimate $c$. Multiplying $Y_{i,2}$ on both sides of equation (\ref{omega1}) and integrating to obtain
\begin{equation}\label{c1}
   \begin{split}
   c \left\langle \sum\limits_{j=1}^{k}Z_{j,1} , Y_{i,2}  \right\rangle & = \left\langle -\Delta \omega_{1} + \mu\epsilon^2 \omega_{1} - p(PV)^{p-1}\omega_{2} , Y_{i,2} \right\rangle - \left\langle h_1 , Y_{i,2} \right\rangle \\
   & = \left\langle -\Delta Y_{i,2} + \mu\epsilon^2 Y_{i,2} , \omega_{1}\right\rangle - \left\langle  p(PV)^{p-1}Y_{i,2}, \omega_{2} \right\rangle - \left\langle h_1 , Y_{i,2} \right\rangle \\
   & = \left\langle qU_{\frac{1}{\Lambda}, x_i}^{q-1}\dfrac{\partial U_{\Lambda , x_i}}{\partial \Lambda} , \omega_{1}\right\rangle  - \left\langle  p(PV)^{p-1}Y_{i,2}, \omega_{2} \right\rangle - \left\langle h_1 , Y_{i,2} \right\rangle.
   \end{split}
\end{equation}
Multiplying $Y_{i,1}$ on both sides of equation (\ref{omega2}) and integrating, we have
\begin{equation}\label{c2}
c \left\langle \sum\limits_{j=1}^{k}Z_{j,2} , Y_{i,1}  \right\rangle = \left\langle pV_{\frac{1}{\Lambda}, x_i}^{p-1}\dfrac{\partial V_{\Lambda , x_i}}{\partial \Lambda} , \omega_{2}\right\rangle  - \left\langle  q(PU)^{q-1}Y_{i,1}, \omega_{1} \right\rangle - \left\langle h_2 , Y_{i,1} \right\rangle.
\end{equation}
By adding equations (\ref{c1}) and (\ref{c2}), we obtain
\[
   \begin{split}
& c\left( \left\langle \sum\limits_{j=1}^{k}Z_{j,1} , Y_{i,2}  \right\rangle +  \left\langle \sum\limits_{j=1}^{k}Z_{j,2} , Y_{i,1}  \right\rangle \right)  = p \left\langle V_{\frac{1}{\Lambda}, x_i}^{p-1}\dfrac{\partial V_{\frac{1}{\Lambda} , x_i}}{\partial \Lambda} - (PV)^{p-1}\dfrac{\partial PV_{\Lambda, x_i}}{\partial \Lambda} , \omega_2 \right\rangle \\
& + q \left\langle U_{\frac{1}{\Lambda}, x_i}^{q-1}\dfrac{\partial U_{\frac{1}{\Lambda} , x_i}}{\partial \Lambda} - (PU)^{q-1}\dfrac{\partial PU_{\Lambda, x_i}}{\partial \Lambda} , \omega_1 \right\rangle -\left\langle h_1 , Y_{i,2} \right\rangle - \left\langle h_2 , Y_{i,1} \right\rangle.
   \end{split}
\]
We estimate each term on both sides separately. Firstly, by Lemma \ref{B2}, we have
\[
\begin{split}
   | \left\langle h_1 , Y_{i,2} \right\rangle | & \leq C ||h_{1}||_{**,1} \int_{\R^N} |Y_{i,2}| \sum\limits_{j=1}^{k} \dfrac{1}{(1+|y-x_j|)^{\frac{N}{q+1}+ 2+\tau}} dy \\
   & \leq C ||h_{1}||_{**,1} \int_{\R^N} \dfrac{1}{(1+|y-x_i|)^{N-2}}\sum\limits_{j=1}^{k} \dfrac{1}{(1+|y-x_j|)^{\frac{N}{q+1}+ 2+\tau}} dy \\
   & \leq C||h_1||_{**,1},
\end{split}
\]
and similarly,
$$
| \left\langle h_2 , Y_{i,1} \right\rangle | \leq C||h_2||_{**,2}.
$$
Then, by Lemma \ref{B1} and Lemma \ref{LemmaA-4}, we have
\[
   \begin{split}
|\omega_2(y)| & \leq C||\omega_{2}||_{*,2} \sum\limits_{j=1}^{k}\dfrac{1}{(1+|y-x_i|)^{\frac{N}{p+1}+\tau}} \\
& \leq  C||\omega_{2}||_{*,2} \left( 1+\sum\limits_{j=2}^{k} \dfrac{1}{|x_1 - x_j|^{\frac{N}{p+1}+\tau}} \right) \\
& \leq C||\omega_{2}||_{*,2}.
\end{split}
\]
Furthermore, from Lemma A.2 we deduce for $N\geq 6$ that
$$
|\psi_{\Lambda,x_j}(y)| \leq \dfrac{C\varepsilon }{(1+|y-x_j|)^{N-3}} \leq \dfrac{C\varepsilon^{\sigma} }{(1+|y-x_j|)^{N-2-\sigma}},
$$
for arbitrary $\sigma \in (0,1)$ since $\varepsilon < \frac{C}{1+|y-x_j|}$. Then, we have, by Lemma \ref{LemmaA-2}, that
\[
   \begin{split}
      &\quad \left| \left\langle V_{\frac{1}{\Lambda}, x_i}^{p-1}\dfrac{\partial V_{\frac{1}{\Lambda} , x_i}}{\partial \Lambda} - (PV)^{p-1}\dfrac{\partial PV_{\Lambda, x_i}}{\partial \Lambda} , \omega_2 \right\rangle \right| \\
      & \leq \int_{\Omega_{\varepsilon}} \left|  \left( V_{\frac{1}{\Lambda},x_i}^{p-1}  - (PV)^{p-1} \right) \dfrac{\partial PV_{\Lambda, x_i}}{\partial \Lambda} \omega_2   + V_{\frac{1}{\Lambda},x_i}^{p-1} \dfrac{\partial \psi_{\Lambda,x_i}}{\partial \Lambda} \omega_2  \right| \\
      & \leq C \int_{\Omega_{\varepsilon}} \left(  \sum\limits_{j \neq i}  (PV_{\Lambda, x_j})^{p-1} \left| \dfrac{\partial PV_{\Lambda, x_i}}{\partial \Lambda}  \right| +  V_{\frac{1}{\Lambda},x_i}^{p-1}\left| \dfrac{\partial \psi_{\Lambda,x_i}}{\partial \Lambda} \right| + \left| \psi_{\Lambda,x_i} \right|^{p-1} \left| \dfrac{\partial PV_{\Lambda,x_i}}{\partial \Lambda} \right| \right) | \omega_2 |.
   \end{split}
\]
Note that by Lemma B.2 and Lemma \ref{LemmaA-4}, we have
\[
   \begin{split}
      &\quad  \int_{\Omega_{\varepsilon}} \left( V_{\frac{1}{\Lambda},x_i}^{p-1}|\dfrac{\partial \psi_{\Lambda,x_i}}{\partial \Lambda} | + | \psi_{\Lambda,x_i}|^{p-1} | \dfrac{\partial PV_{\Lambda,x_i}}{\partial \Lambda} |  \right) |\omega_2| \\
      & \leq C ||\omega_2||_{*,2} \int_{\Omega_{\varepsilon}} \frac{\varepsilon^{\sigma}}{(1+|y-x_i|)^{(N-2)p-\sigma}} + C ||\omega_2||_{*,2} \int_{\Omega_{\varepsilon}}  \dfrac{\varepsilon^{\sigma(p-1)}}{(1+|y-x_i|)^{(N-2)p-\sigma(p-1)}} \\
      & =o(1)||\omega_2||_{*,2}.
   \end{split}
\]
Using the same argument, we have
\[
  \begin{split}
   &\quad  \int_{\Omega_{\varepsilon}} \left| \sum\limits_{j \neq i}  (PV_{\Lambda, x_j})^{p-1}  \dfrac{\partial PV_{\Lambda, x_i}}{\partial \Lambda}   \omega_2 \right| \\
   & \leq C ||\omega_2||_{*,2} \sum\limits_{j\neq i} \int_{\Omega_{\varepsilon}}  \dfrac{1}{(1+|y-x_j|)^{(N-2)(p-1)}} \dfrac{1}{(1+|y-x_i|)^{N-2}} \\
   & \leq C ||\omega_2||_{*,2} \sum\limits_{j\neq i} \dfrac{1}{|x_j-x_i|^{\tau + \sigma } } \int_{\Omega_{\varepsilon}}  \dfrac{1}{(1+|y-x_i|)^{(N-2)p - \tau -\sigma}} \\
   & = o(1)||\omega_2||_{*,2},
  \end{split}
\]
where the last equality holds due to the fact that $p$ satisfies condition (A) and $\sigma$ can be chosen sufficiently small. Therefore, we have
$$
p \left\langle V_{\frac{1}{\Lambda}, x_i}^{p-1}\dfrac{\partial V_{\frac{1}{\Lambda} , x_i}}{\partial \Lambda} - (PV)^{p-1}\dfrac{\partial PV_{\Lambda, x_i}}{\partial \Lambda} , \omega_2 \right\rangle = o(1)||\omega_2||_{*,2}.
$$

As for $N=5$, we can similarly get
\[
   \begin{split}
      &\quad \left| \left\langle V_{\frac{1}{\Lambda}, x_i}^{p-1}\dfrac{\partial V_{\frac{1}{\Lambda} , x_i}}{\partial \Lambda} - (PV)^{p-1}\dfrac{\partial PV_{\Lambda, x_i}}{\partial \Lambda} , \omega_2 \right\rangle \right|  \\
      & \leq C   \int_{\Omega_{\varepsilon}}  \left( V_{\frac{1}{\Lambda},x_i}^{p-2} \sum\limits_{j \neq i} V_{\frac{1}{\Lambda},x_j} \left| \dfrac{\partial PV_{\Lambda,x_i}}{\partial \Lambda} \right| + \left( \sum\limits_{j \neq i} V_{\frac{1}{\Lambda},x_j} \right)^{p-1} \left| \dfrac{\partial PV_{\Lambda,x_i}}{\partial \Lambda} \right|   \right)|\omega_2|   \\
      & +  C \int_{\Omega_{\varepsilon}} \left( V_{\frac{1}{\Lambda},x_i}^{p-1}  \left| \dfrac{\partial \psi_{\Lambda,x_i}}{\partial \Lambda} \right| + \left| \psi_{\Lambda,x_i} \right|^{p-1} \left| \dfrac{\partial PV_{\Lambda,x_i}}{\partial \Lambda} \right| + V_{\frac{1}{\Lambda},x_i}^{p-2} \left| \psi_{\Lambda,x_i} \dfrac{\partial PV_{\Lambda,x_i}}{\partial \Lambda} \right|  \right) | \omega_2 | \\
      & \leq C \int_{\Omega_{\varepsilon}}  \left( \sum\limits_{j \neq i} V_{\frac{1}{\Lambda},x_j} \right)^{p-1} \left| \dfrac{\partial PV_{\Lambda,x_i}}{\partial \Lambda} \right| |\omega_2| + o(1)||\omega_2||_{*,2}
   \end{split}
\]
Let
\[
      \Omega_{j} = \{ y = (y^{\prime}, y^{\prime \prime}) \in \R^2 \times \R^{N-2}: \langle \dfrac{y^{\prime}}{|y^{\prime}|}, \dfrac{x_j}{|x_j|}\rangle \geq cos \dfrac{\pi}{k}  \}.
\]
If $y \in \Omega_i$, then it follows that $|y-x_j| \geq |y - x_i|$ and $2|y-x_j| \geq |x_j - x_i|$ for $j \neq i$. Thus, we obtain
\[
   \begin{split}
      \sum\limits_{j \neq i} V_{\frac{1}{\Lambda}, x_j} & \leq \dfrac{1}{(1+|y-x_i|)^{N-2-\tau- \theta}} \sum\limits_{j \neq i} \dfrac{1}{(1+|y-x_j|^{\tau + \theta})} \\
      & \leq \dfrac{1}{(1+|y-x_i|)^{N-2-\tau- \theta}} \sum\limits_{j \neq i} \dfrac{C}{(|x_i-x_j|^{\tau + \theta})} \\
      & = o(1) \dfrac{1}{(1+|y-x_i|)^{N-2 - \tau - \theta}},
   \end{split}
\]
where $\theta$ can be chosen arbitrarily small. On the other hand, we can deduce from Lemma \ref{LemmaA-4} that
\[
   \begin{split}
       \sum\limits_{j=1}^{k} \dfrac{1}{(1+|y-x_j|)^{\frac{N}{p+1} + \tau}} & \leq \dfrac{1}{(1+|y-x_i|)^{\frac{N}{p+1} + \tau}} + \sum\limits_{j \neq i} \dfrac{C}{|x_j - x_i|^{\tau}} \dfrac{1}{(1+|y-x_j|)^{\frac{N}{p+1}}} \\
       & \leq \dfrac{C}{(1+|y-x_i|)^{\frac{N}{p+1}}}.
   \end{split}
\]
Therefore, it follows that
\[
   \begin{split}
      & \quad \int_{\Omega_i} \left( \sum\limits_{j \neq i} V_{\frac{1}{\Lambda} , x_j}  \right)^{p-1} \left| \dfrac{\partial PV_{\Lambda,x_i}}{\partial \Lambda} \right| |\omega_2| \\
      & =o(1)||\omega_2||_{*,2} \int_{\Omega_1} \dfrac{1}{(1+|y-x_i|)^{(p-1)(N-2-\tau-\theta) + \frac{N}{p+1} + \tau}} \\
      & = o(1)||\omega_2||_{*,2}.
   \end{split}
\]
If $y \in \Omega_{l}$, where $l \neq i$. Then it follows from the same argument as above that
\[
   \begin{split}
      \sum\limits_{j \neq i} V_{\frac{1}{\Lambda}, x_j} & \leq \sum\limits_{j \neq i} \dfrac{1}{(1+|y-x_j|)^{N-2}} \\
      & \leq \dfrac{1}{(1+|y-x_l|)^{N-2}} + \sum\limits_{j \neq l} \dfrac{1}{|x_j - x_l|^{\tau}} \dfrac{1}{(1+|y-x_l|)^{N-2-\tau}} \\
      & \leq \dfrac{C}{(1+|y-x_l|)^{N-2-\tau}}.
   \end{split}
\]
Similarly, we have
\[
   \sum\limits_{j \neq i} \dfrac{1}{(1+|y-x_j|)^{\frac{N}{p+1} + \tau}} \leq  \dfrac{C}{(1+|y-x_l|)^{\frac{N}{p+1}}}.
\]
Therefore, we deduce that
\[
   \begin{split}
      & \quad \int_{\Omega_l} \left( \sum\limits_{j \neq i} V_{\frac{1}{\Lambda} , x_j}  \right)^{p-1} \left| \dfrac{\partial PV_{\Lambda,x_i}}{\partial \Lambda} \right| |\omega_2| \\
      & \leq C \int_{\Omega_l} \dfrac{1}{(1+|y-x_l|)^{(N-2-\tau)(p-1) + \frac{N}{p+1}}} \dfrac{1}{(1+|y-x_i|)^{N-2}} \\
      & \leq  C \dfrac{1}{|x_l - x_i|^{(N-2-\tau)(p-1) + \frac{N}{p+1} - \tau}}.
   \end{split}
\]
As a result, we obtain that
\[
   \begin{split}
       & \quad \int_{\Omega_{\varepsilon}} \left( \sum\limits_{j \neq i} V_{\frac{1}{\Lambda} , x_j}  \right)^{p-1} \left| \dfrac{\partial PV_{\Lambda,x_i}}{\partial \Lambda} \right| |\omega_2| \\
       & = \left( \int_{\Omega_i} + \sum\limits_{l \neq i} \int_{\Omega_l}   \right) \left( \sum\limits_{j \neq i} V_{\frac{1}{\Lambda} , x_j}  \right)^{p-1} \left| \dfrac{\partial PV_{\Lambda,x_i}}{\partial \Lambda} \right| |\omega_2| \\
       & \leq o(1)||\omega_2||_{*,2} + \sum\limits_{j \neq i} \dfrac{C}{|x_l - x_i|^{(N-2-\tau)(p-1) + \frac{N}{p+1} - 2}}\\
       & = o(1)||\omega_2||_{*,2},
   \end{split}
\]
where the last equality holds beacuse $(N-2-\tau)(p-1) + \frac{N}{p+1} - 2 > \tau$ when $N=5$.
In conclusion, when $N \geq 5$ and $p$ satisfies condition (A), it holds that
\[
   p \left\langle V_{\frac{1}{\Lambda}, x_i}^{p-1}\dfrac{\partial V_{\frac{1}{\Lambda} , x_i}}{\partial \Lambda} - (PV)^{p-1}\dfrac{\partial PV_{\Lambda, x_i}}{\partial \Lambda} , \omega_2 \right\rangle = o(1)||\omega_2||_{*,2}.
\]

Similarly, we have
$$
q \left\langle U_{\frac{1}{\Lambda}, x_i}^{q-1}\dfrac{\partial U_{\frac{1}{\Lambda} , x_i}}{\partial \Lambda} - (PU)^{q-1}\dfrac{\partial PU_{\Lambda, x_i}}{\partial \Lambda} , \omega_1 \right\rangle = o(1)||\omega_1||_{*,1}.
$$
The same estimates holds when $N=5$ by employing the same argument. Besides, using a similar computation of Lemma \ref{LE-1}, we know there exists a $\tilde{c}>0$ such that
$$
\left( \left\langle \sum\limits_{j=1}^{k}Z_{j,1} , Y_{i,2}  \right\rangle +  \left\langle \sum\limits_{j=1}^{k}Z_{j,2} , Y_{i,1}  \right\rangle \right) = \tilde{c} + o(1).
$$
Thus we conclude that
$$
c = o(||(\omega_1,\omega_2)||_{*}) + O(||(h_1,h_2)||_{**}).
$$
And consequently, we have
$$
||(\omega_1,\omega_2)||_{*} \leq C \left( o(1) + ||(h_1,h_2)||_{**} + \dfrac{\sum\limits_{j=1}^{k}\frac{1}{(1+|y-x_j|)^{\frac{N}{q+1}+\tau+\theta}}}{\sum\limits_{j=1}^{k}\frac{1}{(1+|y-x_j|)^{\frac{N}{q+1}+\tau}}} + \dfrac{\sum\limits_{j=1}^{k}\frac{1}{(1+|y-x_j|)^{\frac{N}{p+1}+\tau+\theta}}}{\sum\limits_{j=1}^{k}\frac{1}{(1+|y-x_j|)^{\frac{N}{p+1}+\tau}}} \right).
$$
Since $||(\omega_1,\omega_2)||_{*} = 1$, we deduce that there is a $R>0$ and $ c_0 > 0$ such that
$$
||\omega_1||_{L^\infty(B_R(x_i))} + ||\omega_2||_{L^\infty(B_R(x_i))} \geq c_0 > 0
$$
for some $i$. But $(\bar{\omega_1}(y), \bar{\omega_2}(y))=(\omega_1(y-x_i), \omega_2(y-x_i))$ converges uniformly in any compact set of $\R^{N}_+$ to a solution $(\Phi, \Psi)$ of
\begin{equation}\label{linear}
   \begin{cases}
   -\Delta \Phi = p V_{\frac{1}{\Lambda},0}^{p-1} \Phi \;\;\;  \hbox{in} \R^N_{+},\\
   -\Delta \Psi = q U_{\frac{1}{\Lambda},0}^{q-1} \Psi  \;\;\;  \hbox{in} \R^N_{+}.
   \end{cases}
\end{equation}
for some $\Lambda$. We extend $(\Phi, \Psi)$ to the entire $\R^N$ by letting
$$(\Phi(x',x_N), \Psi(x',x_N)) = (\Phi(x',-x_N), \Psi(x',-x_N)).$$
Then it follows that $(\Phi, \Psi)$ is perpendicular to the kernel of equation (\ref{linear}). So $\Phi=0, \Psi=0$. This is a contradiction.
\end{proof}

Now concerning the existence of equations (\ref{equ-reduction}), we have the following lemma:

\begin{lemma}\label{existence}
   There exists a $k_0 > 0$ and a constant $C>0$, independent of k, such that for any $k > k_0$ and any $h=(h_1, h_2) \in L^{\infty}(\Omega_{\epsilon}) \times L^{\infty}(\Omega_{\epsilon})$, problem (\ref{equ-reduction}) has a unique solution $\omega = (\omega_1,\omega_2)=L_k(h)$. Besides, we have
   \[
   ||L_k(h)||_{*} \leq C ||h||_{**},\;\;\; |c|\leq C||h||_{**} .
   \]
   Moreover, the map $L_k(h)$ is $C^1$ with respect to $\Lambda$.
\end{lemma}

\begin{proof}
   Note that problem (\ref{equ-reduction}) is equivalent to
   \begin{equation}\label{equivalent-equa}
      \langle (\omega_1 , \omega_2) , e \rangle_{E} = \langle ( p(PV)^{p-1}\omega_2 , q(PU)^{q-1} \omega_1 ) , e \rangle + \langle ( h_1 , h_2 ) , e \rangle, \;\; \forall e = (e_1,e_2) \in E.
   \end{equation}
    As the right side of (\ref{equivalent-equa}) defines two bounded operator on $E$, by Reisz representation theorem, there are  $(T_1(\omega_1, \omega_2) , T_2(\omega_1, \omega_2)), (\widetilde{h_1}, \widetilde{h_2}) \in E$, such that
   \begin{equation}\label{T}
      \langle (T_1(\omega_1 , \omega_2), T_2(\omega_1, \omega_2)) , e \rangle_{E} = \langle ( p(PV)^{p-1}\omega_2 , q(PU)^{q-1} \omega_1 ) , e \rangle, \;\; \forall e= (e_1,e_2) \in E,
   \end{equation}
   and
   \begin{equation}\label{h}
      \langle (\widetilde{h_1}, \widetilde{h_2}) , e \rangle_{E} = \langle ( h_1, h_2 ) , e \rangle, \;\; \forall e = (e_1,e_2) \in E.
   \end{equation}
   As a result, we obtain an operator $(T_1, T_2)$ defined from $E$ to $E$ and we may rewrite equation (\ref{equivalent-equa}) in the operational form:
   \[
      (\omega_1 , \omega_2) = (T_1(\omega_1 , \omega_2), T_2(\omega_1, \omega_2)) + (\widetilde{h_1}, \widetilde{h_2}) \;\; \hbox{in} \;\; E.
   \]
    Choose $e_1 = 0$ or $e_2 = 0$ in (\ref{T}), we can deduce that
    \[
      \langle T_1(\omega_1 , \omega_2) , e_1 \rangle_{H^1(\Omega_{\epsilon})} = \langle p(PV)^{p-1}\omega_2  , e_1 \rangle
    \]
    and
    \[
      \langle T_2(\omega_1 , \omega_2) , e_2 \rangle_{H^1(\Omega_{\epsilon})} = \langle q(PU)^{q-1}\omega_1  , e_2 \rangle.
    \]
   Thus, $T_1(\omega_1 , \omega_2) = T_1(\omega_2)$ depends only on $\omega_2$ and $T_2(\omega_1 , \omega_2) = T_2(\omega_1)$ depends only on $\omega_1$. From the construction of the operator $(T_1, T_2): E \longrightarrow E$, we can decompose it as follows:
   \begin{equation}\label{decompose}
      E \stackrel{i}{\hookrightarrow \hookrightarrow} L^{2}(\Omega_{\epsilon}) \times L^{2}(\Omega_{\epsilon}) \stackrel{P}{\longrightarrow} L^{2}(\Omega_{\epsilon}) \times L^{2}(\Omega_{\epsilon}) \stackrel{\iota}{\hookrightarrow} E^* \stackrel{\cong}{\longrightarrow} E,
   \end{equation}
   where $$i: E \longrightarrow L^{2}(\Omega_{\epsilon}) \times L^{2}(\Omega_{\epsilon}) $$ is inclusion, $$P :  L^{2}(\Omega_{\epsilon}) \times L^{2}(\Omega_{\epsilon}) \longrightarrow L^{2}(\Omega_{\epsilon}) \times L^{2}(\Omega_{\epsilon})$$ is defined by
   \[
   P(\omega_1, \omega_2) = (p(PV)^{p-1}\omega_2, q(PU)^{q-1}\omega_1) ,
   \]
   and $$\iota : L^{2}(\Omega_{\epsilon}) \times L^{2}(\Omega_{\epsilon}) \longrightarrow E^* $$ is inclusion defined by
   \[
       \iota(\omega_1 , \omega_2)(e_1,e_2)  =  \langle (\omega_1 , \omega_2), (e_1,e_2) \rangle, \;\; \forall (e_1,e_2) \in E.
   \]
   Here, $f(g)$ represents the action of the functional $f$ on the element $g$, where $f \in E^*$ and $g \in E$. Besides, the last isomorphism of (\ref{decompose}) is given by Reisz representation. This shows that the operator $(T_1, T_2)$ is compact since the first inclusion is compact and the other maps are continuous. Then, by Fredholm' alternative theorem, we deduce that equation (\ref{equivalent-equa}) has an unique solution $(\omega_1,\omega_2) \in E$ if and only if the homogenous problem
   \[
      (\omega_1, \omega_2) = (T_1(\omega_2) ,T_2(\omega_1) )
   \]
   has only zero solution. Hence we only need to consider the following problem:
  \begin{equation}\label{homogenous}
   \begin{cases}
      L(\omega_1,\omega_2) = c (\sum\limits_{i=1}^{k} Z_{i,1}, \sum\limits_{i=1}^{k} Z_{i,2})\;\;\; \hbox{in} \;\;\; \Omega_{\epsilon},\\
      \frac{\partial \omega_{1}}{\partial n}=0, \frac{\partial \omega_{2}}{\partial n}=0 \;\;\; \hbox{on} \;\;\; \partial\Omega_{\epsilon}, \\
      \omega_1, \omega_2 \in H_s ,\\
      \langle \sum\limits_{i=1}^{k} Z_{i,1}, \omega_1 \rangle + \langle \sum\limits_{i=1}^{k} Z_{i,2}, \omega_2 \rangle = 0.
   \end{cases}
  \end{equation}
  Suppose $(\omega_1,\omega_2)$ is a solution of (2.12). Then by using Lemma 2.1 with $h_k \equiv 0$, we obtain $(\omega_1,\omega_2)=0$. Thus, (\ref{homogenous}) has only trivial solution and our conclusion follows.
\end{proof}

Recall that we aim to  construct a solution of (\ref{equ-2}) with the form $(PU+\omega_1, PV + \omega_2)$, then we need to ensure that $(\omega_1, \omega_2)$ satisfies (\ref{op-1}). Consequently, we consider the following problem:
\begin{equation}\label{Reduction}
   \begin{cases}
      L(\omega_1,\omega_2) = l + N(\omega_1, \omega_2) + c (\sum\limits_{i=1}^{k} Z_{i,1}, \sum\limits_{i=1}^{k} Z_{i,2})\;\;\; \hbox{in} \;\;\; \Omega_{\epsilon},\\
      \dfrac{\partial \omega_{1}}{\partial n}=\dfrac{\partial \omega_{2}}{\partial n}=0 \;\;\; \hbox{on} \;\;\; \partial\Omega_{\epsilon}, \\
      \omega_1, \omega_2 \in H_s ,\\
      \langle ( \sum\limits_{i=1}^{k} Z_{i,1}, \sum\limits_{i=1}^{k} Z_{i,2} ) , ( \omega_1 ,\omega_2  ) \rangle = 0.
   \end{cases}
\end{equation}

We can use the contraction mapping theorem to prove the existence of the solution to problem (\ref{Reduction}) when $||(\omega_1, \omega_2)||_{*}$ is sufficiently small. For this purpose, we have the following lemma:

\begin{lemma}\label{N}
   Suppose $N\geq 5$ and $p$ satisfies condition (A), then we have
   \[
   ||N(\omega_1, \omega_2)||_{**}  \leq C ||(\omega_1, \omega_2)||_{*}^{min\{ p,2 \}}.
   \]
\end{lemma}

\begin{proof}
   Recall that $N_1(\omega_2) = (PV + \omega_2)^{p} - (PV)^p - p (PV)^{p-1}\omega_2$. Thus, we have
\[ |N_1(\omega_2)| \leq
\begin{cases}
   C|\omega_2|^p, \;\;\; \hbox{if}\;\; p \leq 2 ;\\
   C(PV)^{p-2}\omega_2^2 + C|\omega_2|^p \;\; \hbox{if}\;\; p > 2.
\end{cases}
\]
If $p\leq 2$, let $x = \frac{1}{p}(\frac{N}{q+1} + 2 +\tau)$ and $y=\frac{N}{p+1} + \tau - x $. By using Young inequality, we obtain
\[
   \begin{split}
      |N_1(\omega_2)| & \leq C||\omega_2||_{*,2}^{p} \left( \sum\limits_{j=1}^{k} \dfrac{1}{(1+|y-x_j|)^{x+y}} \right)^p \\
      & \leq C ||\omega_2||_{*,2}^{p} \sum\limits_{j=1}^{k} \dfrac{1}{(1+|y-x_j|)^{xp}} \left( \sum\limits_{j=1}^{k} \dfrac{1}{(1+|y-x_j|)^{yp^{\prime}}} \right)^{\frac{p}{p^{\prime}}} \\
      & = C ||\omega_2||_{*,2}^{p} \sum\limits_{j=1}^{k} \dfrac{1}{(1+|y-x_j|)^{\frac{N}{q+1} + 2 +\tau}} \left( \sum\limits_{j=1}^{k} \dfrac{1}{(1+|y-x_j|)^{\tau}} \right)^{\frac{p}{p^{\prime}}} ,
   \end{split}
\]
where $p^{\prime}=\frac{p}{p-1}$ and the last equality holds once we observe that $yp^{\prime} = \tau$.

By using Lemma \ref{B1}, we have
\[
   \sum\limits_{j=1}^{k} \dfrac{1}{(1+|y-x_j|)^{\tau}} \leq C + C \sum\limits_{j=2}^{k} \dfrac{1}{(|x_1-x_j|)^{\tau}} \leq C.
\]
Hence,
\[
   |N_1(\omega_2)| \leq C||\omega_2||_{*,2}^{p} \sum\limits_{j=1}^{k} \dfrac{1}{(1+|y-x_j|)^{\frac{N}{q+1} + 2 +\tau}}.
\]
Then, we estimate $N_1(\omega_2)$ when $p>2$. Note that this case occurs only when $N=5$ and $p \in (2, \frac{7}{3} ]$. Therefore, we have $p-2<1$, and thus
\[
   \begin{split}
      |N_1(\omega_2)| & \leq C |PV|^{p-2}|\omega_2|^2 + C|\omega_2|^p \\
       & \leq C ||\omega_2||_{*,2}^{2} \left( \sum\limits_{j=1}^k \dfrac{1}{(1+|y-x_j|)^{N-2}} \right) ^{p-2} \left( \sum\limits_{j=1}^k \dfrac{1}{(1+|y-x_j|)^{\frac{N}{q+1}+\tau}} \right)^2 \\
       & + C||\omega_2||_{*,2}^{p} \sum\limits_{j=1}^{k} \dfrac{1}{(1+|y-x_j|)^{\frac{N}{q+1} + 2 +\tau}} \\
       & \leq C ||\omega_2||_{*,2}^{2} \left( \sum\limits_{j=1}^k \dfrac{1}{(1+|y-x_j|)^{\frac{N}{q+1}+\tau}} \right)^p + C||\omega_2||_{*,2}^{p} \sum\limits_{j=1}^{k} \dfrac{1}{(1+|y-x_j|)^{\frac{N}{q+1} + 2 +\tau}} \\
       & \leq C(||\omega_2||_{*,2}^{2} + ||\omega_2||_{*,2}^{p}) \sum\limits_{j=1}^{k} \dfrac{1}{(1+|y-x_j|)^{\frac{N}{q+1} + 2 +\tau}} .
   \end{split}
\]
Here we assume $||(\omega_1, \omega_2)||_{*}$ is small enough, then it follows that
\[
   ||N_1(\omega_2)||_{**,1} \leq C||\omega_2||_{*,2}^{min(p,2)}.
\]
Note that when $N\geq 5$ and $p$ satisfies condition (A), this is always true that $q\leq 3$. Therefore, interchanging the role of $p$ and $q$ yields the following result.
\[
   ||N_2(\omega_1)||_{**,2} \leq C||\omega_1||_{*,1}^{min(q,2)}.
\]
Given that $q \geq p$ and $||(\omega_1, \omega_2)||_{*}$ is sufficiently small, we can deduce the following conclusion:
\[
   ||N(\omega_1, \omega_2)||_{**} \leq C||(\omega_1,\omega_2)||_{*}^{min(p,2)}.
\]
\end{proof}

Next, we will proceed with the estimation of $||(l_1,l_2)||_{**}$.

\begin{lemma}\label{l}
   Suppose $N\geq 5$ and $p$ satisfies condition (A), then we have
   \[
   ||(l_1, l_2)||_{**} \leq C \epsilon^{\frac{1}{2} + \sigma}.
   \]
where $\sigma > 0$ is a fixed small constant.
\end{lemma}

\begin{proof}
   Recall that
   \[
      \Omega_{j} = \{ y = (y^{\prime}, y^{\prime \prime}) \in \R^2 \times \R^{N-2}: \langle \dfrac{y^{\prime}}{|y^{\prime}|}, \dfrac{x_j}{|x_j|}\rangle \geq cos \dfrac{\pi}{k}  \}.
   \]
   Without loss of generality, we may assume $y \in \Omega_1$. Hence, we have $|y-x_j| \geq |y-x_1|$ and we need to estimate $||l_1||_{**,1}$ and $||l_2||_{**,2}$ respectively.

   Recall that $l = (l_1, l_2) = ( (PV)^p - \sum\limits_{j=1}^{k}V_{\frac{1}{\Lambda}, x_j}^p, (PU)^q - \sum\limits_{j=1}^{k}U_{\frac{1}{\Lambda}, x_j}^q  ) $.Therefore, when $N\geq 6$, we have
   \begin{equation}\label{l1}
      \begin{split}
         &\quad |(PV)^{p} -  \sum\limits_{j=1}^{k}V_{\frac{1}{\Lambda}, x_j}^p |\\
         &  = |(PV_{\Lambda,x_1} + \sum\limits_{j=2}^{k}PV_{\Lambda, x_j})^p - \sum\limits_{j=1}^{k}V_{\frac{1}{\Lambda}, x_j}^p| \\
         & = |PV_{\Lambda,x_1}^p + p PV_{\Lambda,x_1}^{p-1}\sum\limits_{j=2}^{k}(PV_{\Lambda,x_j}) + O((\sum\limits_{j=2}^{k} PV_{\Lambda,x_j})^p) - \sum\limits_{j=1}^{k} V_{\frac{1}{\Lambda},x_j}^p | \\
         & \leq C |(PV_{\Lambda,x_1}^p - V_{\frac{1}{\Lambda},x_1}^p )| + C |PV_{\Lambda,x_1}|^{p-1} \sum\limits_{j=2}^{k} |PV_{\Lambda,x_j}| \\
         & + C \sum\limits_{j=2}^{k} | PV_{\Lambda, x_j} |^p + C \left( \sum\limits_{j=2}^{k} |V_{\frac{1}{\Lambda},x_j}| \right) ^p \\
         & \leq C  \dfrac{1}{(1+|y-x_1|)^{(N-2)(p-1)}}  |\psi_{\Lambda,x_1} | +  C |\psi_{\Lambda,x_1}|^p  + C \sum\limits_{j=2}^{k} \left( \dfrac{1}{(1+|y-x_j|)^{N-2}} \right)^p \\
         & + C \dfrac{1}{(1+|y-x_1|)^{(N-2)(p-1)}} \sum\limits_{j=2}^{k} \dfrac{1}{(1+|y-x_j|)^{N-2}} .
      \end{split}
   \end{equation}
   Note that
   \[
      \begin{split}
      & \quad\dfrac{1}{(1+|y-x_1|)^{(N-2)(p-1)}} \sum\limits_{j=2}^{k} \dfrac{1}{(1+|y-x_j|)^{N-2}} \\
      & \leq C \sum\limits_{j=2}^{k} \dfrac{1}{|x_1-x_j|^{\frac{Np}{q+1}-\tau}} \left( \dfrac{1}{(1+|y-x_1|)^{(N-2)p - \frac{Np}{q+1}+ \tau}} + \dfrac{1}{(1+|y-x_j|)^{(N-2)p - \frac{Np}{q+1}+ \tau}}\right) \\
      & \leq C \dfrac{1}{(1+|y-x_1|)^{\frac{N}{q+1} + \tau +2}} \sum\limits_{j=2}^k \dfrac{1}{|x_1-x_j|^{\frac{Np}{q+1} - \tau}} \\
      & \leq C (\epsilon k)^{\frac{Np}{q+1} - \tau} \sum\limits_{j=1}^k\dfrac{1}{(1+|y-x_k|)^{\frac{N}{q+1} + \tau +2}} \\
      & \leq C \epsilon^{\frac{1}{2}+\sigma} \sum\limits_{j=1}^k\dfrac{1}{(1+|y-x_k|)^{\frac{N}{q+1} + \tau +2}},
   \end{split}
   \]
   where the last inequality holds since $\frac{Np}{q+1} - \tau > \frac{N-2}{2}$ when $p$ satisfies condition (A) .

   On the other hand, we have
   \[
      \begin{split}
         & \quad \dfrac{1}{(1+|y-x_j|)^{N-2}}  \leq \dfrac{1}{(1+|y-x_1|)^{\frac{N-2}{2}}}\dfrac{1}{(1+|y-x_j|)^{\frac{N-2}{2}}} \\
         & \leq C \dfrac{1}{|x_1- x_j|^{\frac{N}{q+1}-\frac{\tau}{p}}} \left(  \dfrac{1}{(1+|y-x_1|)^{N-2-\frac{N}{q+1}+\frac{\tau}{p}}} + \dfrac{1}{(1+|y-x_j|)^{N-2-\frac{N}{q+1}+\frac{\tau}{p}}}  \right) \\
         & \leq C \dfrac{1}{|x_1- x_j|^{\frac{N}{q+1}-\frac{\tau}{p}}} \dfrac{1}{(1+|y-x_1|)^{N-2-\frac{N}{q+1}+\frac{\tau}{p}}}.
      \end{split}
   \]
 Since
\[
   \left( \sum\limits_{j=2}^{k} \dfrac{1}{|x_1- x_j|^{\frac{N}{q+1}-\frac{\tau}{p}}} \right)^p \leq C\epsilon^{\frac{Np}{q+1}-\tau} k^p = C \epsilon^{\frac{Np}{q+1} - \tau(p+1)} \leq C \epsilon^{\frac{1}{2} + \sigma},
\]
we conclude
   \[
   \begin{split}
      \sum\limits_{j=2}^{k} \left( \dfrac{1}{(1+|y-x_j|)^{N-2}} \right)^p & \leq C \left( \sum\limits_{j=2}^{k} \dfrac{1}{|x_1- x_j|^{\frac{N}{q+1}-\frac{\tau}{p}}} \right)^p \dfrac{1}{(1+|y-x_1|)^{(N-2)p-\frac{Np}{q+1}+\tau}} \\
      & = C  \epsilon^{\frac{1}{2}+\sigma} \dfrac{1}{(1+|y-x_1|)^{\frac{N}{q+1}+\tau + 2}}.
   \end{split}
   \]
Besides, by using Lemma \ref{LemmaA-2}, we estimate the first and second term of on the right side of (\ref{l1}):
   \[
      \begin{split}
      \dfrac{1}{(1+|y-x_1|)^{(N-2)(p-1)}}  |\psi_{\Lambda,x_1} | & \leq C \dfrac{\epsilon^{\frac{Np}{q+1}-\tau}}{(1+|y-x_1|)^{(N-2)p-\frac{Np}{q+1}+\tau}} \\
      & \leq C \epsilon^{\frac{1}{2}+\sigma} \sum\limits_{j=1}^{k} \dfrac{1}{(1+|y-x_j|)^{\frac{N}{q+1}+2+\tau}}.
   \end{split}
   \]
   Similarly, we obtain
   \[
      |\psi_{\Lambda,x_1} |^p \leq C \epsilon^{\frac{1}{2}+\sigma} \sum\limits_{j=1}^{k} \dfrac{1}{(1+|y-x_j|)^{\frac{N}{q+1}+2+\tau}}.
   \]
   Thus, we have
   \[
   ||l_1||_{**,1} \leq C\epsilon^{\frac{1}{2}+\sigma} .
   \]
   When $N=5$ and $p \in (2,\frac{7}{3}]$, by the similar arguments, we have
   \[
   \begin{split}
     & \quad  |(PV)^{p} -  \sum\limits_{j=1}^{k}V_{\frac{1}{\Lambda}, x_j}^p | \\
     &  \leq C  \dfrac{1}{(1+|y-x_1|)^{(N-2)(p-1)}}  |\psi_{\Lambda,x_1} |+ C \dfrac{1}{(1+|y-x_1|)^{(N-2)(p-2)}}  |\psi_{\Lambda,x_1} |^2 \\
      & +  C |\psi_{\Lambda,x_1}|^p  + C \sum\limits_{j=2}^{k} \dfrac{1}{(1+|y-x_j|)^{(N-2)p}}   + C \dfrac{1}{(1+|y-x_1|)^{(N-2)(p-1)}} \sum\limits_{j=2}^{k} \dfrac{1}{(1+|y-x_j|)^{N-2}}\\
      &  + C \dfrac{1}{(1+|y-x_1|)^{(N-2)(p-2)}} \left( \sum\limits_{j=2}^{k} \dfrac{1}{(1+|y-x_j|)^{N-2}} \right)^2  \\
      & \leq C \epsilon^{\frac{1}{2}+\sigma} \sum\limits_{j=1}^{k} \dfrac{1}{(1+|y-x_j|)^{\frac{N}{q+1}+2+\tau}} + C \dfrac{1}{(1+|y-x_1|)^{(N-2)(p-2)}}  |\psi_{\Lambda,x_1} |^2 \\
      & + C \dfrac{1}{(1+|y-x_1|)^{(N-2)(p-2)}}  \sum\limits_{j=2}^{k} \dfrac{1}{(1+|y-x_j|)^{2(N-2)}}.
   \end{split}
   \]
   Note that
   \[
       \begin{split}
         \dfrac{1}{(1+|y-x_1|)^{(N-2)(p-2)}}  |\psi_{\Lambda,x_1} |^2 &  \leq C \dfrac{\epsilon^{\frac{Np}{q+1}-\tau}}{(1+|y-x_1|)^{(N-2)p-\frac{Np}{q+1}+\tau}} \\
         & \leq C \epsilon^{\frac{1}{2}+\sigma} \sum\limits_{j=1}^{k} \dfrac{1}{(1+|y-x_j|)^{\frac{N}{q+1}+2+\tau}},
       \end{split}
   \]
   and
   \[
      \begin{split}
         \dfrac{1}{(1+|y-x_1|)^{(N-2)(p-2)}} & \sum\limits_{j=2}^{k} \dfrac{1}{(1+|y-x_j|)^{2(N-2)}} \\
         & \leq C \sum\limits_{j=2}^{k} \dfrac{1}{|x_1-x_j|^{\frac{Np}{q+1}-\tau}} \dfrac{1}{(1+|y-x_1|)^{(N-2)p-\frac{Np}{q+1}+\tau}} \\
         & \leq C \epsilon^{\frac{1}{2}+\sigma} \sum\limits_{j=1}^{k} \dfrac{1}{(1+|y-x_j|)^{\frac{N}{q+1}+2+\tau}}.
      \end{split}
   \]
   Consequently, we  have
   \[
      ||l_1||_{**,1} \leq C\epsilon^{\frac{1}{2}+\sigma}.
   \]
   Similarly, we get
   \[
      ||l_2||_{**,2} \leq C\epsilon^{\frac{1}{2}+\sigma}.
   \]
   This completes the proof.
\end{proof}

After obtaining the estimates for $\|l\|_{**}$ and $\|N(\omega_1,\omega_2)\|_{**}$, we can utilize the contraction mapping theorem to establish the existence of the solution to equation (\ref{Reduction}) for sufficiently large value of $k$.

\begin{theorem}
   Suppose $N\geq 5$ and $p$ satisfies condition (A). Then, there exists a $k_0>0$, such that for each $k>k_0$, $\delta \leq \Lambda \leq \delta^{-1}$, where $\delta$ is a small fixed constant, equations (\ref{Reduction}) has an unique solution $(\omega_1,\omega_2)$, which satisfies
   \[
      ||(\omega_1,\omega_2)||_{*} \leq C\epsilon^{\frac{1}{2}+\sigma},
   \]
   where $\sigma > 0$ is a small constant. Besides, the map $\Lambda \rightarrow ((\omega_1(\Lambda),\omega_2(\Lambda)))$ is $C^1$.
\end{theorem}

\begin{proof}
   Define
   \[
      E_{N} = \{ (u,v) \in E, ||(u,v)||_{*} \leq \epsilon^{\frac{1}{2}+\sigma} \}.
   \]
   In order to find a solution of (\ref{Reduction}) in $E_{N}$. We note that equation (\ref{Reduction}) is equivalent to
   \[
      (\omega_1,\omega_2) = A(\omega_1,\omega_2) : = L_k(N_1(\omega_2),N_2(\omega_1)) + L_k(l_1,l_2), \;\; \hbox{for} (\omega_1, \omega_2) \in E_N.
   \]
   Hence, it is sufficient to find a fixed point of the operator $A$ in the space $E_N$. The fixed point can be readily obtained if we can demonstrate that $A$ is a contraction map from $E_N$ to $E_N$, which we will prove in the following.

   First, we verify $A$ maps from $E_N$ to $E_N$. For any $(\omega_1,\omega_2) \in E_{N}$, by Lemma \ref{N} and Lemma \ref{l}, we have
   \begin{equation}\label{A}
     \begin{split}
      ||A(\omega_1,\omega_2)||_{*} & \leq ||L_k(N_1(\omega_2),N_2(\omega_1))||_* + ||L_k(l_1,l_2)||_* \\
      & \leq C ||(N_1(\omega_2),N_2(\omega_1))||_* + C ||(l_1,l_2)||_* \\
      & \leq C ||(\omega_1 , \omega_2)||_{*}^{min(p,2)} + C\epsilon^{\frac{1}{2}+\sigma} \\
      & \leq C \epsilon^{\frac{min(p,2)}{2}}  + C\epsilon^{\frac{1}{2}+\sigma} \leq  C\epsilon^{\frac{1}{2}+\sigma}.
     \end{split}
   \end{equation}
  Here $\sigma>0$ can be choosen arbitrarily small, and therefore, we have $A(\omega_1,\omega_2) \in E_N$.

   Next, we prove $A$ is a contraction map. For any $(\omega_1, \omega_2)$ and $(\phi_1,\phi_2) \in E_N$, we have
   \[
      \begin{split}
         ||A(\omega_1,\omega_2) - A(\phi_1,\phi_2)||_* & = ||L_k(N_1(\omega_2),N_2(\omega_1)) - L_k(N_1(\phi_2),N_2(\phi_1))||_*\\
         & \leq C||(N_1(\omega_2) - N_1(\phi_2) ,  N_2(\omega_1) - N_2(\phi_1))||_{**} .
      \end{split}
   \]
   When $N \geq 6$, $p$ satisfies condtion (A), we have
   \[
      |N_1^{\prime}(t)| \leq C |t|^{p-1}.
   \]
   Therefore,
   \[
      \begin{split}
         |N_1(\omega_2) - N_1(\phi_2)| &  \leq C (|\omega_2|^{p-1} + |\phi_2|^{p-1})|\omega_2 - \phi_2| \\
         & \leq C ( ||\omega_2||_{*,2}^{p-1} + ||\phi_2||_{*,2}^{p-1} )||\omega_2 - \phi_2||_{*,2} \left( \sum\limits_{j=1}^{k} \dfrac{1}{(1+|y-x_j|)^{\frac{N}{p+1}+\tau}}  \right)^p.
      \end{split}
   \]
   As is shown in Lemma \ref{N}, we have
   \[
      \left( \sum\limits_{j=1}^{k} \dfrac{1}{(1+|y-x_j|)^{\frac{N}{p+1}+\tau}}  \right)^p \leq C \sum\limits_{j=1}^{k} \dfrac{1}{(1+|y-x_j|)^{\frac{N}{q+1}+2+\tau}}.
   \]
   As a result, we obtain
   \[
      ||N_1(\omega_2) - N_1(\phi_2)||_{**,1} \leq \epsilon^{\sigma}||\omega_2-\phi_2||_{*,2}
   \]
   since $\epsilon$ can be chosen small enough.

   When $N=5$ and $p \in (2,\frac{7}{3}]$, we have
   \[
      | N_1^{\prime}(t) | \leq C(|PV|^{p-2}|t| + |t|^{p-1}).
   \]
   By using similar estimate as in the proof of Lemma \ref{l}, we have
   \[
      \begin{split}
         &\quad  |N_1(\omega_2) - N_1(\phi_2)|   \leq C ((PV)^{p-2}(| \omega_2 | + | \phi_2 | )  + |\omega_2|^{p-1} + |\phi_2|^{p-1})|\omega_2 - \phi_2| \\
         & \leq C (||\omega_2||_{*,2}+||\phi_2||_{*,2}) ||\omega_2 - \phi_2||_{*,2} \left( \sum\limits_{j=1}^{k} \dfrac{1}{(1+|y-x_j|)^{\frac{N}{p+1}+\tau}}  \right)^2 \\
         & \times \left( \left( \sum\limits_{j=1}^{k} \dfrac{1}{(1+|y-x_j|)^{N-2}} \right)^{p-2} +  \left( \sum\limits_{j=1}^{k} \dfrac{1}{(1+|y-x_j|)^{\frac{N}{p+1}+\tau}}  \right)^{p-2} \right)\\
         & \leq C \epsilon^{\frac{1}{2}+\sigma} ||\omega_2 - \phi_2||_{*,2} \sum\limits_{j=1}^k \dfrac{1}{(1+|y-x_j|)^{\frac{N}{q+1}+2+\tau}}.
      \end{split}
   \]
   Choosing $\epsilon$ small enough, we  have
   \[
      ||N_1(\omega_2) - N_1(\phi_2)||_{**,1} \leq \epsilon^{\sigma}||\omega_2-\phi_2||_{*,2}.
   \]
   Similarly, we also  have
   \[
      ||N_2(\omega_1) - N_2(\phi_1)||_{**,2} \leq \epsilon^{\sigma}||\omega_1-\phi_1||_{*,1}.
   \]
   Therefore, we obtain
   \[
   \begin{split}
      ||A(\omega_1,\omega_2) - A(\phi_1,\phi_2)||_* & \leq C||(N_1(\omega_2) - N_1(\phi_2) ,  N_2(\omega_1) - N_2(\phi_1))||_{**} \\
      & \leq C\epsilon^{\sigma} ||(\omega_1 - \phi_1, \omega_2-\phi_2)||_{*} \leq \dfrac{1}{2} |(\omega_1, \omega_2) - (\phi_1, \phi_2)||_{*}.
   \end{split}
   \]
   The last inequality holds if we choose $\epsilon = k^{-\frac{N-2}{N-3}}$ small enough.

   As a result, for large enough $k$, $A$ is a contraction map from $E_N$ to $E_N$ and it follows that there is a unique solution $(\omega_1.\omega_2) \in E_N$ of equation (\ref{Reduction}). Moreover, by (\ref{A}), we have
   \[
   ||(\omega_1,\omega_2)||_{*} = ||A(\omega_1,\omega_2)||_{*} \leq C \epsilon^{\frac{1}{2}+\sigma}.
   \]
\end{proof}

\section{Energy estimates}
In Section 2, we have demonstrated the existence and uniqueness of a solution $(\omega_1, \omega_2)$ to equation (\ref{Reduction}) for each fixed $\Lambda \in (\delta, \delta^{-1})$. Consequently, to find a solution to equation (\ref{op-1}), it suffices to select an appropriate $\Lambda$ such that the constant $c$ in equation (\ref{Reduction}) equals zero.

Note that the equation (\ref{equ-2}) possesses a variational structure that corresponds to the following associated energy functional:
\begin{equation}
   I(u,v)=\int_{\Omega_\varepsilon} (\nabla u \nabla v + \mu \varepsilon^2 uv) - \frac{1}{p+1} \int_{\Omega_\varepsilon} v^{p+1} -\frac{1}{q+1} \int_{\Omega_\varepsilon} u^{q+1}.
\end{equation}
Let
\[
    F(\Lambda) = I (PU+\omega_1(\Lambda), PV+\omega_2(\Lambda)) .
\]
Hence to find a critical point $\Lambda \in (\delta, \delta^{-1})$ of $F(\Lambda)$. Once such $\Lambda$ is found, $(PU+\omega_1(\Lambda), PV+\omega_2(\Lambda))$ is a solution of equations (\ref{equ-2}).

In the following, we will first estimate the energy of $(PU,PV)$ and then estimate $F(\Lambda)$.

For simplicity, denote $\lambda = \frac{1}{\varepsilon \Lambda}$ and $\widetilde{x}_j = \varepsilon x_j$. In order to estimate the energy of $(PU,PV)$, we need to following lemma.

\begin{lemma}\label{LE-1}
It holds that
$$
\int_{\Omega_\varepsilon} V_{\frac{1}{\Lambda},x_j}^{p+1} =\overline{A}_0 - \overline{A}_1 \gamma\Lambda \varepsilon + O(\varepsilon^{2-\sigma}),
$$
and
$$
\int_{\Omega_\varepsilon} U_{\frac{1}{\Lambda},x_j}^{q+1} =\overline{B}_0 - \overline{B}_1 \gamma\Lambda \varepsilon + O(\varepsilon^{2-\sigma}),
$$
where $\overline{A}_0, \overline{A}_1, \overline{B}_0, \overline{B}_1$ are some positive constants, and $\sigma > 0$ is a small constant.
\end{lemma}

\begin{proof}
We have
\begin{equation}\label{en-1}
\begin{split}
\int_{\Omega_\varepsilon} V_{\frac{1}{\Lambda},x_j}^{p+1} =& \int_{\Omega} V_{\frac{1}{\varepsilon \Lambda}, \varepsilon x_j}^{p+1} \\
=& \int_{\Omega \cap B_\delta(\varepsilon x_j)} V_{\lambda, \widetilde{x}_j}^{p+1} + \int_{\Omega \setminus B_\delta(\varepsilon x_j)} V_{\lambda, \widetilde{x}_j}^{p+1} \\
=&  \int_{\Omega \cap B_\delta(\widetilde{x}_j)} V_{\lambda, \widetilde{x}_j}^{p+1} + O\left( \int_{\Omega_\lambda \setminus B_{\delta \lambda} (0) } \frac{1}{(1+|y|)^{(N-2)(p+1)}}  \right) \\
=& \int_{\Omega \cap B_\delta(\widetilde{x}_j)} V_{\lambda, \widetilde{x}_j}^{p+1} + O\left( \frac{1}{\lambda^{(N-2)p-2}} \right).
\end{split}
\end{equation}
For simplicity, we denote $x := \widetilde{x}_j $. Note that $x \in \Gamma \subset \partial \Omega$, then for $\delta > 0$ small, we have

$$
\Omega \cap B_\delta(x) = \{ y=(y',y_N)\in \mathbb R^{N-1} \times \mathbb R \; : \; |y-x| < \delta, \; y_N > f(y')  \},
$$
where
    \begin{equation}\label{en-2}
        f(y')=x_N + \frac{1}{2}\sum\limits_{i=1}^{N-1} k_i (y_i -x_i)^2 + O\left( |y'-x'|^3 \right),
    \end{equation}
and $k_i$ is the principal curvature at $x,\; i=1,\cdots,N-1$. Besides, the mean curvature at $x$ is defined as the mean value of the principal curvature at $x$:
    \begin{equation}\label{en-3}
        H(x) =\frac{1}{N-1}\sum\limits_{i=1}^{N-1}k_i \equiv \gamma.
    \end{equation}
Then, we can divide the first integral of the right side of (\ref{en-1}) into three parts:
    \begin{equation}\label{en-4}
        \int_{\Omega \cap B_\delta(x)} V_{\lambda, x}^{p+1} = \int_{B_\delta^+(x)} V_{\lambda, x}^{p+1} - \int_{W'} V_{\lambda, x}^{p+1} + \int_{W''} V_{\lambda, x}^{p+1}: = I_1 + I_2 + I_3,
    \end{equation}
where
\[
    \begin{split}
        B_\delta^+(x):=&\{ y\; : \; |y-x|<\delta, y_N > x_N \}, \\
        W':=&\{ y\; : \; |y-x|<\delta, x_N < y_N < f(y')\}, \\
        W'':=&\{ y\; : \; |y-x|<\delta, f(y') < y_N < x_N\}.
    \end{split}
\]
Then,
    \begin{equation}\label{en-5}
        \begin{split}
           I_1 & = \frac{1}{2} \int_{B_\delta(x)} V_{\lambda, x}^{p+1} \\
           &  = \frac{1}{2} \int_{B_{\lambda \delta}(0)} V_{1,0}^{p+1} \\
           & = \frac{1}{2} \int_{\mathbb R^N} V_{1,0}^{p+1} + O\left( \frac{1}{\lambda^{(N-2)p-2}} \right)\\
           & =  \int_{\mathbb R^N_+} V_{1,0}^{p+1} + O\left( \frac{1}{\lambda^{(N-2)p-2}} \right).
       \end{split}
    \end{equation}

We denote $B_{\delta}(x') :=\{ y\in \mathbb R^{N-1} \; : \; \sum\limits_{i=1}^{N-1} (y_i - x_i)^2  < \delta^2\}$, which is a ball in $\mathbb R^{N-1}$. Then, using \eqref{en-2}, we have
   \begin{equation}\label{en-6}
     \begin{split}
       &\int_{W'} V_{\lambda, x}^{p+1} -\int_{W'} V_{\lambda, x}^{p+1} = \int_{B_\delta(x')} dy' \int_{x_N}^{f(y')} V_{\lambda, x}^{p+1} dy_N \\
        =& \int_{B_\delta(x')} V_{\lambda,x'}^{p+1}(y') (f(y')-x_N)dy' + O\left( \int_{B_{\delta}(x')} \frac{\lambda^{N-1}}{ (1+\lambda |y'-x'|)^{(N-2)(p+1)-1} } \frac{\lambda^3(f(y')-x_N)^3}{(1+\lambda |y'-x'|)^3}       \right) \\
       =& \int_{B_{\delta\lambda}(0')} \lambda [(f(\lambda^{-1}y' + x')-x_N)] V_{1,0'}^{p+1} dy' + O\left( \int_{B_{\delta\lambda}(0')} \frac{\lambda^3\left[f(\lambda^{-1}y'+x')-x_N\right]^3}{(1+|y'|)^{(N-2)p+N}}       \right) \\
       =& \int_{B_{\delta\lambda}(0')} \frac{\lambda}{2} \left[ \sum\limits_{i=1}^{N-1} k_i \left(  \frac{y_i}{\lambda} \right)^2 + O \left( \left|  \frac{y'}{\lambda}  \right|^3    \right) \right] V_{1,0'}^{p+1} dy' + O\left( \int_{B_{\delta\lambda}(0')} \frac{\lambda^3 \left| \frac{y'}{\lambda}  \right|^6}{(1+|y'|)^{(N-2)p+N}}       \right) \\
      =& \frac{(N-1)H(x)}{2\lambda} \int_{B_{\delta\lambda}(0')} y_1^2 V_{1,0'}^{p+1} dy' + O\left( \frac{1}{\lambda^2} \int_{B_{\delta\lambda}(0')} \frac{|y'|^3}{(1+|y'|)^{(N-2)(p+1)}}  \right)   +O\left( \frac{1}{\lambda^3}\right) \\
      =& \frac{\gamma}{2\lambda} \int_{B_{\delta\lambda}(0')} |y'|^2 V_{1,0'}^{p+1} dy' + O\left( \frac{1}{\lambda^{2-\sigma}} \right) \\
      =& \frac{\gamma}{2\lambda} \int_{\mathbb R^{N-1}} |y'|^2 V_{1,0'}^{p+1} dy' + O\left( \frac{1}{\lambda^{2-\sigma}} \right).
      \end{split}
   \end{equation}
Combining \eqref{en-1}, \eqref{en-4}, \eqref{en-5} and \eqref{en-6}, we  obtain
\[
\begin{split}
\int_{\Omega_\varepsilon} V_{\frac{1}{\Lambda},x_j}^{p+1} &= \int_{\mathbb R^N_+} V_{1,0}^{p+1} - \frac{\gamma}{2\lambda} \int_{\mathbb R^{N-1}} |y'|^2 V_{1,0'}^{p+1} dy' + O\left( \frac{1}{\lambda^{2-\sigma}} \right) \\
&=\overline{A}_0 - \overline{A}_1 \gamma\Lambda \varepsilon + O(\varepsilon^{2-\sigma}).
\end{split}
\]
We can use a similar argument to give the estimate of $\displaystyle \int_{\Omega_\varepsilon} U_{\frac{1}{\Lambda},x_j}^{q+1}$, and the result thus follows.
\end{proof}

\begin{lemma}\label{LE-2}
It holds that
$$
\int_{\Omega_\varepsilon} V_{\frac{1}{\Lambda},x_j}^{p} \psi_{\Lambda,x_j} = -\overline{A}_3 \gamma \Lambda \varepsilon + O\left( \varepsilon^{2-\sigma} \right),
$$
and
$$
\int_{\Omega_\varepsilon} U_{\frac{1}{\Lambda},x_j}^{q} \varphi_{\Lambda,x_j} = -\overline{B}_3 \gamma \Lambda \varepsilon + O\left( \varepsilon^{2-\sigma} \right),
$$
where $\overline{A}_3, \overline{B}_3$ are some positive constants.

\end{lemma}

\begin{proof}
Using Lemma \ref{LemmaA-1} and by the similar argument as in the proof of Lemma \ref{LE-1}, we have
\[
\begin{split}
  &\int_{\Omega_\varepsilon} V_{\frac{1}{\Lambda},x_j}^{p} \psi_{\Lambda,x_j} =\int_{\Omega_\varepsilon} V_{\frac{1}{\Lambda},x_j}^{p}  \varepsilon \Lambda^{1-\frac{N}{p+1}} \varphi_0\left(\frac{y-x_j}{\Lambda} \right) \\
  + &  O\left( \int_{\Omega_\varepsilon}
  \frac{1}{(1+|y-x_j|)^{(N-2)p}}  \left[ \frac{\varepsilon^{2}|\ln \varepsilon|^m}{  (1+| y-x_j|)^{N-4}    }  +  \varepsilon^{N-2} \right] \right) \\
  =&\int_{\Omega_\varepsilon} V_{\frac{1}{\Lambda},x_j}^{p}  \varepsilon \Lambda^{1-\frac{N}{p+1}} \varphi_0\left(\frac{y-x_j}{\Lambda} \right) + O\left( \varepsilon^{2-\sigma} \right) \\
  =& (\varepsilon \Lambda)^{1-\frac{N}{p+1}} \int_{\Omega} V_{\frac{1}{\varepsilon \Lambda},\varepsilon x_j}^{p} \varphi_0\left(\frac{y-\varepsilon x_j}{\varepsilon \Lambda} \right) + O\left( \varepsilon^{2-\sigma} \right) \\
  =& \lambda^{\frac{N}{p+1} - 1} \int_{\Omega} V_{\lambda,\widetilde{x}_j}^{p} \varphi_0 (\lambda (y-\widetilde{x}_j)) + O\left( \varepsilon^{2-\sigma} \right) \\
  =& \lambda^{\frac{N}{p+1} - 1} \int_{\Omega \cap B_\delta (\widetilde{x}_j)} V_{\lambda,\widetilde{x}_j}^{p} \varphi_0 (\lambda (y-\widetilde{x}_j)) + O\left( \varepsilon^{2-\sigma} \right) \\
  =& \frac{1}{\lambda} \int_{\mathbb R^N_+} V_{1,0}^{p} \varphi_0  + O\left( \varepsilon^{2-\sigma} \right) \\
  =& \frac{1}{\lambda} \int_{\mathbb R^N_+} (-\Delta U_{1,0} \varphi_0 + U_{1,0} \Delta \varphi_0) + O\left( \varepsilon^{2-\sigma} \right) \\
  =& \frac{1}{\lambda} \int_{\partial \mathbb R^N_+} \frac{\partial \varphi_0}{\partial n}   U_{1,0} + O\left( \varepsilon^{2-\sigma} \right) \\
  =& -\frac{1}{\lambda} \frac{N-2}{2} \sum\limits_{i=1}^{N-1} k_i \int_{\partial \mathbb R^N_+}  U_{1,0} \frac{y_i^2}{(1+|y|^2)^{\frac{N}{2}}} + O\left( \varepsilon^{2-\sigma} \right) \\
  =& -\frac{1}{\lambda} \frac{N-2}{2} \gamma \int_{\partial \mathbb R^N_+}  U_{1,0} \frac{|y|^2}{(1+|y|^2)^{\frac{N}{2}}} + O\left( \varepsilon^{2-\sigma} \right) \\
  =& -\overline{A}_3 \gamma \Lambda \varepsilon + O\left( \varepsilon^{2-\sigma} \right).
\end{split}
\]
Similarly, we can obtain
$$
\int_{\Omega_\varepsilon} U_{\frac{1}{\Lambda},x_j}^{q} \varphi_{\Lambda,x_j} = -\overline{B}_3 \gamma \Lambda \varepsilon + O\left( \varepsilon^{2-\sigma} \right).
$$
\end{proof}

\begin{lemma}\label{LE-3}
It holds that
\[
  \begin{split}
    &\int_{\Omega_\varepsilon} \left( \nabla PU_{\Lambda, x_j} \nabla PV_{\Lambda, x_j} + \mu \varepsilon^2  PU_{\Lambda, x_j} PV_{\Lambda, x_j} \right) \\
    =&  \frac{\overline{A}_0+\overline{B}_0}{2} + (\frac{\overline{A}_3+\overline{B}_3}{2} - \frac{\overline{A}_1+\overline{B}_1}{2}) \gamma \Lambda \varepsilon + O(\varepsilon^{2-\sigma}),
   \end{split}
\]
where the constant $\overline{A_i}$ and $\overline{B_i}$ are given in Lemma \ref{LE-1} and Lemma \ref{LE-2}.
\end{lemma}

\begin{proof}
Using Lemma \ref{LE-1} and Lemma \ref{LE-2}, by directly calculating, we have
\[
   \begin{split}
     &\int_{\Omega_\varepsilon} \left( \nabla PU_{\Lambda, x_j} \nabla PV_{\Lambda, x_j} + \mu \varepsilon^2  PU_{\Lambda, x_j} PV_{\Lambda, x_j} \right) \\
     =&\frac{1}{2}\int_{\Omega_\varepsilon} \left[ \left( -\Delta PU_{\Lambda, x_j} + \mu \varepsilon^2  PU_{\Lambda, x_j} \right) PV_{\Lambda, x_j}
     + \left( -\Delta PV_{\Lambda, x_j} + \mu \varepsilon^2  PV_{\Lambda, x_j} \right) PU_{\Lambda, x_j}  \right]\\
     =& \frac{1}{2} \int_{\Omega_\varepsilon}\left[ V^p_{\frac{1}{\Lambda},x_j} PV_{\Lambda, x_j} + U^q_{\frac{1}{\Lambda},x_j} PU_{\Lambda, x_j}\right] \\
     =&  \frac{1}{2} \int_{\Omega_\varepsilon} \left[ V^p_{\frac{1}{\Lambda},x_j} (V_{\frac{1}{\Lambda}, x_j}-\psi_{\Lambda, x_j}) +
     U^q_{\frac{1}{\Lambda},x_j} (U_{\frac{1}{\Lambda}, x_j}-\varphi_{\Lambda, x_j}) \right] \\
     =& \frac{\overline{A}_0+\overline{B}_0}{2} + (\frac{\overline{A}_3+\overline{B}_3}{2} - \frac{\overline{A}_1+\overline{B}_1}{2}) \gamma \Lambda \varepsilon + O(\varepsilon^{2-\sigma}).
   \end{split}
\]
\end{proof}

\begin{lemma}\label{LE-4}
It holds that
$$
\frac{1}{p+1} \int_{\Omega_\varepsilon} (PV_{\Lambda, x_j} )^{p+1} = \frac{1}{p+1} \overline{A}_0 +\left( \overline{A}_3 -\frac{1}{p+1} \overline{A}_1 \right) \gamma \Lambda \varepsilon + O(\varepsilon^{2-\sigma}),
$$
and
$$
\frac{1}{q+1} \int_{\Omega_\varepsilon} (PU_{\Lambda, x_j} )^{q+1} = \frac{1}{q+1} \overline{B}_0 +\left( \overline{B}_3 -\frac{1}{q+1} \overline{B}_1 \right) \gamma \Lambda \varepsilon + O(\varepsilon^{2-\sigma}).
$$
\end{lemma}

\begin{proof}
Using Lemma \ref{LE-1} and Lemma \ref{LE-2}, we obtain
\[
\begin{split}
&\frac{1}{p+1} \int_{\Omega_\varepsilon} (PV_{\Lambda, x_j} )^{p+1} \\
=& \frac{1}{p+1} \int_{\Omega_\varepsilon} ( V_{\frac{1}{\Lambda}, x_j} -\psi_{\Lambda, x_j}  )^{p+1} \\
=& \frac{1}{p+1} \int_{\Omega_\varepsilon} V_{\frac{1}{\Lambda}, x_j}^{p+1} - \int_{\Omega_\varepsilon} V_{\frac{1}{\Lambda}, x_j}^{p} \psi_{\Lambda, x_j}
+O\left(   \int_{\Omega_\varepsilon} V_{\frac{1}{\Lambda}, x_j}^{p-1}   \psi_{\Lambda, x_j}^2  + \int_{\Omega_\varepsilon}  \psi_{\Lambda, x_j}^{p+1} \right) \\
=& \frac{1}{p+1} \overline{A}_0 +\left( \overline{A}_3 -\frac{1}{p+1} \overline{A}_1 \right) \gamma \Lambda \varepsilon + O(\varepsilon^{2-\sigma}).
\end{split}
\]
The estimate of $\frac{1}{q+1}\displaystyle  \int_{\Omega_\varepsilon} (PU_{\Lambda, x_j} )^{q+1}$ can be obtained by using a similar arguments.

\end{proof}

\begin{lemma}\label{LE-5}
For $i \neq j$, it holds
$$
\int_{\Omega_\varepsilon} V_{\frac{1}{\Lambda},x_i}^{p}V_{\frac{1}{\Lambda},x_j} =\frac{\overline{A}_4\Lambda^{N-2}}{|x_i-x_j|^{N-2}} + O\left(\frac{1}{|x_i-x_j|^{(N-2)p-2}}\right),
$$
and
$$
\int_{\Omega_\varepsilon} U_{\frac{1}{\Lambda},x_i}^{q}U_{\frac{1}{\Lambda},x_j} =\frac{\overline{B}_4\Lambda^{N-2}}{|x_i-x_j|^{N-2}} + O\left(\frac{1}{|x_i-x_j|^{N}}\right),
$$
where $\overline{A}_4, \overline{B}_4$ are some positive constants.

\end{lemma}

\begin{proof}
We have
\begin{equation}\label{en-7}
    \begin{split}
      &\int_{\Omega_\varepsilon} V_{\frac{1}{\Lambda},x_i}^{p}V_{\frac{1}{\Lambda},x_j}  = \int_{\Omega} V_{\frac{1}{\varepsilon \Lambda}, \varepsilon x_i}^{p} V_{\frac{1}{\varepsilon \Lambda}, \varepsilon x_j} \\
      =&\int_{\Omega\setminus (B_d(\widetilde{x}_i) \cup B_d(\widetilde{x}_j) )} V_{\lambda, \widetilde{x}_i}^{p} V_{\lambda, \widetilde{x}_j}
      +\int_{\Omega \cap B_d(\widetilde{x}_i)} V_{\lambda, \widetilde{x}_i}^{p} V_{\lambda, \widetilde{x}_j}
      + \int_{\Omega \cap B_d(\widetilde{x}_j)} V_{\lambda, \widetilde{x}_i}^{p} V_{\lambda, \widetilde{x}_j} \\
      :=&I_1+I_2+I_3,
    \end{split}
\end{equation}
where we choose $d=\frac{1}{2} |\widetilde{x}_i - \widetilde{x}_j|$.

We first estimate $I_1$:
\begin{equation}\label{en-8}
\begin{split}
I_1 =& \int_{\Omega\setminus (B_d(\widetilde{x}_i) \cup B_d(\widetilde{x}_j) )} V_{\lambda, \widetilde{x}_i}^{p} V_{\lambda, \widetilde{x}_j} \\
\leq & C\lambda^N \int_{\Omega\setminus (B_d(\widetilde{x}_i) \cup B_d(\widetilde{x}_j) )} \left( \frac{1}{(\lambda |y-\widetilde{x}_i|)^{(N-2)(p+1)}}  +  \frac{1}{(\lambda |y-\widetilde{x}_j|)^{(N-2)(p+1)}}  \right) \\
\leq & \frac{C}{(\lambda d)^{(N-2)p-2}} \leq  \frac{C}{|x_i-x_j|^{(N-2)p-2}}.
\end{split}
\end{equation}

As for $I_2$, we have:
\begin{equation}\label{en-9}
\begin{split}
I_2 =& \int_{\Omega \cap B_d(\widetilde{x}_i)} V_{\lambda, \widetilde{x}_i}^{p} V_{\lambda, \widetilde{x}_j} = \int_{\widetilde{\Omega}_\lambda \cap B_{\lambda d}(0)} V_{1, 0}^{p} V_{1, \lambda (\widetilde{x}_j - \widetilde{x}_i)} \\
=& \frac{1}{\lambda^{N-2} |\widetilde{x}_j - \widetilde{x}_i|^{N-2}} \int_{\widetilde{\Omega}_\lambda \cap B_{\lambda d}(0)} V_{1, 0}^{p}
+O\left( \frac{1}{\lambda^{N} |\widetilde{x}_j - \widetilde{x}_i|^{N}} \int_{\widetilde{\Omega}_\lambda \cap B_{\lambda d}(0)} V_{1, 0}^{p} \right) \\
=& \frac{\Lambda^{N-2}}{|x_j-x_i|^{N-2}} \int_{\widetilde{\Omega}_\lambda \cap B_{\lambda d}(0)} V_{1, 0}^{p}
+O\left( \frac{1}{|x_j - x_i|^{N}} \int_{\widetilde{\Omega}_\lambda \cap B_{\lambda d}(0)} V_{1, 0}^{p} \right) \\
=& \frac{\Lambda^{N-2}}{|x_j-x_i|^{N-2}} \int_{\mathbb R^N_+} V_{1, 0}^{p}  + +O\left( \frac{1}{|x_j - x_i|^{N}} \right) .\\
\end{split}
\end{equation}

Then, for $I_3$, we have
\begin{equation}\label{en-10}
\begin{split}
I_3 =& \int_{\Omega \cap B_d(\widetilde{x}_j)} V_{\lambda, \widetilde{x}_i}^{p} V_{\lambda, \widetilde{x}_j} \\
=& \int_{\widetilde{\Omega}_\lambda \cap B_{\lambda d}(0)}  V_{1, \lambda (\widetilde{x}_i - \widetilde{x}_j)}^p V_{1, 0}\\
\leq & \frac{C}{(\lambda |\widetilde{x}_i-\widetilde{x}_j|)^{(N-2)p}} \int_{ B_{\lambda d}(0)} \frac{1}{(1+|y|)^{N-2}} \\
\leq & \frac{C(\lambda d)^2}{(\lambda |\widetilde{x}_i-\widetilde{x}_j|)^{(N-2)p}} \\
\leq &\frac{C}{|x_i-x_j|^{(N-2)p-2}}.
\end{split}
\end{equation}
Combining \eqref{en-8}, \eqref{en-9} and \eqref{en-10} and note that $p$ satisfies condition (A), we obtain
\[
\begin{split}
\int_{\Omega_\varepsilon} V_{\frac{1}{\Lambda},x_i}^{p}V_{\frac{1}{\Lambda},x_j} =& \frac{\Lambda^{N-2}}{|x_j-x_i|^{N-2}} \int_{\mathbb R^N_+} V_{1, 0}^{p}
 + O\left(\frac{1}{|x_i-x_j|^{(N-2)p-2}}\right) \\
=& \frac{\overline{A}_4\Lambda^{N-2}}{|x_j-x_i|^{N-2}} + O\left(\frac{1}{|x_i-x_j|^{(N-2)p-2}}\right).
\end{split}
\]
Note that $q\geq \frac{N+2}{N-2}$, we can similarly prove the estimate of $\displaystyle\int_{\Omega_\varepsilon} U_{\frac{1}{\Lambda},x_i}^{q}U_{\frac{1}{\Lambda},x_j} $.

\end{proof}

\begin{lemma}\label{LE-6}
For $i \neq j$, we have
$$
   \int_{\Omega_\varepsilon} V_{\frac{1}{\Lambda},x_i}^{p} \psi_{\Lambda,x_j} = O\left( \frac{ \varepsilon |\ln \varepsilon|^m}{|x_i-x_j|^{N-3}} \right),
$$
and
$$
   \int_{\Omega_\varepsilon} U_{\frac{1}{\Lambda},x_i}^{q} \varphi_{\Lambda,x_j} = O\left( \frac{ \varepsilon |\ln \varepsilon|^m}{|x_i-x_j|^{N-3}} \right),
$$
where $m=1$ for $N=5$ and $m=0$ for $N \geq 6$.
\end{lemma}

\begin{proof}

Using Lemma \ref{LemmaA-2}, we have
\begin{equation}\label{en-11}
   \begin{split}
     &\int_{\Omega_\varepsilon} V_{\frac{1}{\Lambda},x_i}^{p} \psi_{\Lambda,x_j} \leq C \int_{\Omega_\varepsilon} \frac{1}{(1+|y-x_i|)^{(N-2)p}} \frac{\varepsilon |\ln \varepsilon|^m}{(1+|y-x_j|)^{N-3}} \\
     \leq & C \frac{|\ln \varepsilon|^m}{\varepsilon^{N-1}} \int_{\Omega} \frac{1}{(1+\varepsilon^{-1}|y-\varepsilon x_i|)^{(N-2)p}} \frac{1}{(1+\varepsilon^{-1}|y-\varepsilon x_j|)^{N-3}} \\
     \leq & C \left( \int_{\Omega\setminus (B_d(\widetilde{x}_i) \cup B_d(\widetilde{x}_j) )}
     + \int_{\Omega \cap B_d(\widetilde{x}_i)} + \int_{\Omega \cap B_d(\widetilde{x}_j)}   \right)  \frac{|\ln \varepsilon|^m \varepsilon^{1-N}}{(1+\varepsilon^{-1}|y-\widetilde{x}_i|)^{(N-2)p}} \frac{1}{(1+\varepsilon^{-1}|y- \widetilde{x}_j|)^{N-3}} \\
     := & I_1+I_2+I_3,
   \end{split}
\end{equation}
where $d=\frac{1}{2} |\widetilde{x}_i - \widetilde{x}_j|$.

First, we estimate $I_1$:
\begin{equation}\label{en-12}
\begin{split}
I_1 =&\frac{|\ln \varepsilon|^m}{\varepsilon^{N-1}} \int_{\Omega\setminus (B_d(\widetilde{x}_i) \cup B_d(\widetilde{x}_j) )}
\frac{1}{(1+\varepsilon^{-1}|y-\widetilde{x}_i|)^{(N-2)p}} \frac{1}{(1+\varepsilon^{-1}|y- \widetilde{x}_j|)^{N-3}} \\
\leq & \frac{C |\ln \varepsilon|^m}{\varepsilon^{N-1}}  \int_{\Omega\setminus (B_d(\widetilde{x}_i) \cup B_d(\widetilde{x}_j) )}
\left(  \frac{1}{ (\varepsilon^{-1} |y-\widetilde{x}_i|)^{(N-2)p+N-3}  } +    \frac{1}{ (\varepsilon^{-1} |y-\widetilde{x}_j|)^{(N-2)p+N-3}  }  \right) \\
\leq & \frac{C  |\ln \varepsilon|^m  \varepsilon^{(N-2)p-2}   }{d^{(N-2)p-3}}
\leq \frac{C\varepsilon |\ln \varepsilon|^m}{|x_i-x_j|^{(N-2)p-3}}.
\end{split}
\end{equation}

For $I_2$, we have
\begin{equation}\label{en-13}
\begin{split}
I_2 =& \frac{|\ln \varepsilon|^m}{\varepsilon^{N-1}} \int_{\Omega \cap B_d(\widetilde{x}_i)} \frac{1}{(1+\varepsilon^{-1}|y-\widetilde{x}_i|)^{(N-2)p}} \frac{1}{(1+\varepsilon^{-1}|y- \widetilde{x}_j|)^{N-3}} \\
= &  \varepsilon |\ln \varepsilon|^m \int_{\widetilde{\Omega}_\varepsilon \cap B_{\frac{d}{\varepsilon}}(0)}
\frac{1}{(1+|y|)^{(N-2)p}} \frac{1}{(1+|y+\varepsilon^{-1} \widetilde{x}_i - \varepsilon^{-1} \widetilde{x}_j|)^{N-3}}\\
\leq & \frac{C \varepsilon |\ln \varepsilon|^m}{ ( \varepsilon^{-1}|\widetilde{x}_i -\widetilde{x}_j|)^{N-3}} \leq  \frac{C \varepsilon |\ln \varepsilon|^m}{|x_i-x_j|^{N-3}}.
\end{split}
\end{equation}

Then, for $I_3$, we have
\begin{equation}\label{en-14}
\begin{split}
I_3 =& \frac{|\ln \varepsilon|^m}{\varepsilon^{N-1}} \int_{\Omega \cap B_d(\widetilde{x}_j)} \frac{1}{(1+\varepsilon^{-1}|y-\widetilde{x}_i|)^{(N-2)p}} \frac{1}{(1+\varepsilon^{-1}|y- \widetilde{x}_j|)^{N-3}} \\
=&\varepsilon |\ln \varepsilon|^m \int_{\widetilde{\Omega}_\varepsilon \cap B_{\frac{d}{\varepsilon}}(0)}
\frac{1}{(1+|y+\varepsilon^{-1} \widetilde{x}_j - \varepsilon^{-1} \widetilde{x}_i|)^{(N-2)p}} \frac{1}{(1+|y|)^{N-3}} \\
\leq & \frac{C\varepsilon |\ln \varepsilon|^m }{(\varepsilon^{-1} |\widetilde{x}_j - \widetilde{x}_i|)^{(N-2)p}} \int_{B_{\frac{d}{\varepsilon}}(0)} \frac{1}{(1+|y|)^{N-3}}  \\
\leq & \frac{C\varepsilon |\ln \varepsilon|^m (d \varepsilon^{-1})^3 }{ |x_j - x_i|^{(N-2)p}} \leq  \frac{C\varepsilon |\ln \varepsilon|^m}{|x_i-x_j|^{(N-2)p-3}}.
\end{split}
\end{equation}
Combining \eqref{en-12}, \eqref{en-13} and \eqref{en-14}, we obtain

$$
\int_{\Omega_\varepsilon} V_{\frac{1}{\Lambda},x_i}^{p} \psi_{\Lambda,x_j} = O\left( \frac{ \varepsilon |\ln \varepsilon|^m}{|x_i-x_j|^{N-3}} \right).
$$
The estimate of $\displaystyle\int_{\Omega_\varepsilon} U_{\frac{1}{\Lambda},x_i}^{q} \varphi_{\Lambda,x_j}$ can be obtained similarly.
\end{proof}

\begin{lemma}\label{LE-7}
For $i \neq j$, it holds that
\[
  \begin{split}
    &\int_{\Omega_\varepsilon} \left( \nabla PU_{\Lambda, x_i} \nabla PV_{\Lambda, x_j} + \mu \varepsilon^2  PU_{\Lambda, x_i} PV_{\Lambda, x_j} \right)\\
    =& \frac{\overline{A}_4+\overline{B}_4}{2} \frac{\Lambda^{N-2}}{|x_i-x_j|^{N-2}} + O\left( \frac{1}{|x_i-x_j|^{(N-2)p-2}} + \frac{\varepsilon |\ln \varepsilon|^m}{|x_i-x_j|^{N-3}}    \right).
   \end{split}
\]
\end{lemma}

\begin{proof}

Using Lemma \ref{LE-5} and Lemma \ref{LE-6}, we have
\[
\begin{split}
   &\int_{\Omega_\varepsilon} \left( \nabla PU_{\Lambda, x_i} \nabla PV_{\Lambda, x_j} + \mu \varepsilon^2  PU_{\Lambda, x_i} PV_{\Lambda, x_j} \right) \\
   =&\frac{1}{2} \int_{\Omega_\varepsilon} \left[\left( -\Delta  PU_{\Lambda, x_i} + \mu \varepsilon^2  PU_{\Lambda, x_i}   \right) PV_{\Lambda, x_j}
   + \left( -\Delta  PV_{\Lambda, x_j} + \mu \varepsilon^2  PV_{\Lambda, x_j}   \right) PU_{\Lambda, x_i}\right] \\
   =& \frac{1}{2} \int_{\Omega_\varepsilon} \left[ V^p_{\frac{1}{\Lambda},x_i} (V_{\frac{1}{\Lambda},x_j} -\psi_{\Lambda, x_j})
   +U^q_{\frac{1}{\Lambda},x_j} (U_{\frac{1}{\Lambda},x_i} -\varphi_{\Lambda, x_i}) \right]\\
   =& \frac{\overline{A}_4+\overline{B}_4}{2} \frac{\Lambda^{N-2}}{|x_i-x_j|^{N-2}} + O\left( \frac{1}{|x_i-x_j|^{(N-2)p-2}} + \frac{\varepsilon |\ln \varepsilon|^m}{|x_i-x_j|^{N-3}}    \right).
\end{split}
\]

\end{proof}

\begin{lemma}\label{LE-8}
For $i \neq j$, it holds that
$$
\int_{\Omega_\varepsilon} (PV_{\Lambda,x_i})^{p}PV_{\Lambda,x_j} =\frac{\overline{A}_4\Lambda^{N-2}}{|x_i-x_j|^{N-2}} + O\left(\frac{1}{|x_i-x_j|^{(N-2)p-2}} + \frac{\varepsilon |\ln \varepsilon|^{m(p+1)}}{|x_i-x_j|^{N-3}}\right),
$$
and
$$
\int_{\Omega_\varepsilon} (PU_{\Lambda,x_i})^{q}PU_{\Lambda,x_j} =\frac{\overline{B}_4\Lambda^{N-2}}{|x_i-x_j|^{N-2}} + O\left(\frac{1}{|x_i-x_j|^{N}} + \frac{\varepsilon |\ln \varepsilon|^{m(q+1)}}{|x_i-x_j|^{N-3}}\right).
$$
\end{lemma}

\begin{proof}

Similar to the proof of Lemma \ref{LE-6}, we have
\begin{equation}
   \begin{split}
   &\int_{\Omega_\varepsilon} (PV_{\Lambda,x_i})^{p}PV_{\Lambda,x_j} \\
   =&\int_{\Omega_\varepsilon} V_{\frac{1}{\Lambda},x_i}^{p} PV_{\Lambda,x_j}
   +O\left( \int_{\Omega_\varepsilon} V_{\frac{1}{\Lambda},x_i}^{p-1} \psi_{\Lambda,x_i} PV_{\Lambda,x_j}  +
   \int_{\Omega_\varepsilon} \psi_{\Lambda,x_i}^p PV_{\Lambda,x_j}   \right) \\
   =&\int_{\Omega_\varepsilon} V_{\frac{1}{\Lambda},x_i}^{p} PV_{\Lambda,x_j}
   +O\left( \int_{\Omega_\varepsilon} \frac{|\ln \varepsilon|^{mp}}{(1+|y-x_i|)^{(N-2)p}}  \frac{\varepsilon |\ln \varepsilon|^{m}}{(1+|y-x_j|)^{N-3}}   \right) \\
   =& \int_{\Omega_\varepsilon} V_{\frac{1}{\Lambda},x_i}^{p} (V_{\frac{1}{\Lambda},x_j} - \psi_{\Lambda,x_j})
   +O\left(   \frac{\varepsilon |\ln \varepsilon|^{m(p+1)}}{|x_i-x_j|^{N-3}}    \right) \\
   =& \int_{\Omega_\varepsilon} V_{\frac{1}{\Lambda},x_i}^{p} V_{\frac{1}{\Lambda},x_j} + O\left(   \frac{\varepsilon |\ln \varepsilon|^{m(p+1)}}{|x_i-x_j|^{N-3}}    \right) \\
   =& \frac{\overline{A}_4\Lambda^{N-2}}{|x_i-x_j|^{N-2}} + O\left(\frac{1}{|x_i-x_j|^{(N-2)p-2}} + \frac{\varepsilon |\ln \varepsilon|^{m(p+1)}}{|x_i-x_j|^{N-3}}  \right).
\end{split}
\end{equation}
Moreover, the estimate of $\int_{\Omega_\varepsilon} (PU_{\Lambda,x_i})^{q}PU_{\Lambda,x_j}$ can be obtained by the same manner.

\end{proof}

Now we are readily to estimate the energy of $(PU,PV).$
\begin{proposition}\label{prop-energy}
We have
$$
I(PU,PV)=k\left( Q_0 -Q_1 \gamma \Lambda \varepsilon -Q_4 \Lambda^{N-2} \varepsilon +o(\varepsilon) \right),
$$
where $Q_0,Q_1,Q_4$ are positive constants, and $\gamma$ is the mean curvature of $\partial \Omega$ along $\Gamma$.
\end{proposition}

\begin{proof}
We have
\begin{equation}\label{en-p1}
  \begin{split}
   &I(PU,PV)\\
   =&\int_{\Omega_\varepsilon} (\nabla PU \nabla PV + \mu \varepsilon^2 PU \cdot PV) - \frac{1}{p+1} \int_{\Omega_\varepsilon} (PV)^{p+1} -\frac{1}{q+1} \int_{\Omega_\varepsilon} (PU)^{q+1}.
  \end{split}
\end{equation}

By symmetry and Lemma \ref{LE-3} and Lemma \ref{LE-7}, we have

\[
\begin{split}
  &\int_{\Omega_\varepsilon} (\nabla PU \nabla PV + \mu \varepsilon^2 PU \cdot PV)\\
  =& k \int_{\Omega_\varepsilon} \left( \nabla PU_{\Lambda, x_1}\nabla PV_{\Lambda, x_1} + \mu \varepsilon^2 PU_{\Lambda, x_1}PV_{\Lambda, x_1} \right) \\
  &+ k \int_{\Omega_\varepsilon} \sum\limits_{j=2}^k \left(  \nabla PU_{\Lambda, x_1}  \nabla PV_{\Lambda, x_j}  + \mu \varepsilon^2 PU_{\Lambda, x_1}PV_{\Lambda, x_j}  \right) \\
  =& k\left( \overline{D}_0 + (\overline{D}_3 - \overline{D}_1)\gamma \Lambda \varepsilon  + O(\varepsilon^{2-\sigma})    \right) \\
  &+k\left( \sum\limits_{j=2}^k \left[\frac{\overline{D}_4  \Lambda^{N-2} }{|x_1-x_j|^{N-2}}  +O\left( \frac{1}{|x_1-x_j|^{(N-2)p-2}} + \frac{\varepsilon |\ln \varepsilon|^m}{|x_1-x_j|^{N-3}}     \right) \right]     \right),
\end{split}
\]
where $\overline{D}_0 =\frac{\overline{A}_0+\overline{B}_0}{2},\;\overline{D}_1 =\frac{\overline{A}_1+\overline{B}_1}{2},\;\overline{D}_3 =\frac{\overline{A}_3+\overline{B}_3}{2},\;\overline{D}_4 =\frac{\overline{A}_4+\overline{B}_4}{2}$ are positive constants.

Using Lemma \ref{LemmaA-4}, we have
\begin{equation}\label{en-important}
   \begin{split}
     &\sum\limits_{j=2}^k \frac{1}{|x_1-x_j|^{(N-2)p-2}} =O\left(  (\varepsilon k)^{(N-2)p-2}  \right) = O\left( \varepsilon^{\frac{(N-2)p-2}{N-2}}    \right) = O\left( \varepsilon^{1+\sigma}  \right),\\
     &\sum\limits_{j=2}^k \frac{\varepsilon |\ln \varepsilon|^m}{|x_1-x_j|^{N-3}} =O\left( \varepsilon |\ln \varepsilon|^m (\varepsilon k)^{N-3} \ln k \right)
     =  O\left( \varepsilon^{1+\sigma}  \right).
   \end{split}
\end{equation}
Thus, it follows that
\begin{equation}\label{en-p2}
   \begin{split}
     &\int_{\Omega_\varepsilon} (\nabla PU \nabla PV + \mu \varepsilon^2 PU \cdot PV)\\
     =&k\left( \overline{D}_0 + (\overline{D}_3 - \overline{D}_1)\gamma \Lambda \varepsilon +   \sum\limits_{j=2}^k \frac{\overline{D}_4  \Lambda^{N-2} }{|x_1-x_j|^{N-2}}  + O\left( \varepsilon^{1+\sigma}  \right)   \right).
   \end{split}
\end{equation}

Next, we give the estimate of $\frac{1}{p+1} \displaystyle\int_{\Omega_\varepsilon} (PV)^{p+1}$. Recall that
$$
\Omega_j = \{ (y',y'')\in \Omega_\varepsilon \; : \; y' \in \mathbb R^2, \langle \frac{(y',0)}{|y'|}, \frac{x_j}{|x_j|}  \rangle \geq \cos \frac{\pi}{k} \},\;\; j=1,\cdots,k.
$$
Then, by symmetry, we have
\begin{equation}\label{en-p3}
  \begin{split}
    &\frac{1}{p+1}  \int_{\Omega_\varepsilon} (PV)^{p+1} = \frac{k}{p+1}  \int_{\Omega_1} (PV)^{p+1} \\
    =& \frac{k}{p+1} \int_{\Omega_1} (PV_{\Lambda,x_1})^{p+1} + k  \int_{\Omega_1} (PV_{\Lambda,x_1})^{p} \sum\limits_{j=2}^k PV_{\Lambda,x_j} + O\left( k \int_{\Omega_1} (PV_{\Lambda,x_1})^{p-1} \left( \sum\limits_{j=2}^k PV_{\Lambda,x_j} \right)^2 \right)\\
    :=& I_1 + I_2 +O(I_3).
  \end{split}
\end{equation}

For $I_1$, using Lemma \ref{LemmaA-2} and Lemma \ref{LE-4}, it follows that
\begin{equation}\label{en-p4}
  \begin{split}
    I_1=& \frac{k}{p+1} \int_{\Omega_1} (PV_{\Lambda,x_1})^{p+1} \\
    =& \frac{k}{p+1} \int_{\Omega_\varepsilon} (PV_{\Lambda,x_1})^{p+1} + O\left(  k \int_{\Omega_\varepsilon \setminus \Omega_1}  \frac{|\ln \varepsilon|^{m(p+1)}}{(1+|y-x_1|)^{(N-2)(p+1)}}   \right) \\
    =& \frac{k}{p+1} \int_{\Omega_\varepsilon} (PV_{\Lambda,x_1})^{p+1} + O\left( k |\ln \varepsilon|^{m(p+1)} (k\varepsilon)^{(N-2)(p+1)-N} \right) \\
    =& \frac{k}{p+1} \int_{\Omega_\varepsilon} (PV_{\Lambda,x_1})^{p+1} + O\left( k \varepsilon^{1+\sigma} \right) \\
    =& k \left[ \frac{1}{p+1} \overline{A}_0 + \left( \overline{A}_3 - \frac{1}{p+1} \overline{A}_1   \right)\gamma \Lambda \varepsilon + O\left(  \varepsilon^{1+\sigma} \right) \right],
  \end{split}
\end{equation}
where the third identity holds because $|y-x_1| \geq \frac{C}{\varepsilon} \sin \frac{\pi}{k} \geq \frac{C}{k\varepsilon}$ for $y\in \Omega_\varepsilon \setminus \Omega_1$.

As for $I_2$, we use Lemma \ref{LemmaA-2} to obtain
\begin{equation}\label{en-p5}
  \begin{split}
    I_2=& k  \int_{\Omega_1} (PV_{\Lambda,x_1})^{p} \sum\limits_{j=2}^k PV_{\Lambda,x_j} \\
    =& k \sum\limits_{j=2}^k \int_{\Omega_\varepsilon} (PV_{\Lambda,x_1})^{p}PV_{\Lambda,x_j} +
    O\left( k \sum\limits_{j=2}^k  \int_{\Omega_\varepsilon \setminus \Omega_1}  \frac{|\ln\varepsilon|^{m(p+1)}}{ (1+|y-x_1|)^{(N-2)p} (1+|y-x_j|)^{N-2} }   \right).
  \end{split}
\end{equation}
We write $\Omega_\varepsilon \setminus \Omega_1 = W_1 \cup W_2 \cup W_3$, where
\[
   \begin{split}
     &W_1 = (\Omega_\varepsilon \setminus \Omega_1) \cap B_d(x_j), \\
     &W_2 = (\Omega_\varepsilon \setminus \Omega_1) \cap B_d(x_1), \\
     &W_3 = (\Omega_\varepsilon \setminus \Omega_1) \setminus (B_d(x_j) \cup B_d(x_1)),
\end{split}
\]
and $d=\frac{1}{2}|x_1-x_j|$. Then, we have
\begin{equation}\label{en-p6}
   \begin{split}
     &\int_{W_3}  \frac{1}{ (1+|y-x_1|)^{(N-2)p} (1+|y-x_j|)^{N-2} } \\
     \leq & C \int_{W_3} \left[ \frac{1}{ (1+|y-x_1|)^{(N-2)(p+1)} } + \frac{1}{ (1+|y-x_j|)^{(N-2)(p+1)}}  \right]\\
     \leq & \frac{C}{d^{(N-2)p-2}} \leq \frac{C}{|x_1-x_j|^{(N-2)p-2}},
   \end{split}
\end{equation}
and
\begin{equation}\label{en-p7}
    \begin{split}
      &\int_{W_1}  \frac{1}{ (1+|y-x_1|)^{(N-2)p} (1+|y-x_j|)^{N-2} } \\
      \leq & \frac{C}{d^{(N-2)p}} \int_{W_1} \frac{1}{(1+|y-x_j|)^{N-2}} \\
      \leq & \frac{C}{d^{(N-2)p}} d^2 \leq \frac{C}{|x_1-x_j|^{(N-2)p-2}}.
    \end{split}
\end{equation}

For $y\in W_2$, we have $\frac{C}{k\varepsilon} \leq |y-x_1| \leq d$. Then, it follows that
\begin{equation}\label{en-p8}
   \begin{split}
     &\int_{W_2}  \frac{1}{ (1+|y-x_1|)^{(N-2)p} (1+|y-x_j|)^{N-2} } \\
     \leq & \frac{C}{d^{N-2}} \int_{W_2} \frac{1}{(1+|y-x_1|)^{(N-2)p}} \\
     \leq & \frac{C}{d^{N-2}} \left[ \left( \frac{1}{k\varepsilon}   \right)^{N-(N-2)p} + d^{N-(N-2)p}   \right] \\
     \leq & \frac{C}{d^{(N-2)p-2}} + \frac{C(k\varepsilon)^{(N-2)p-N}}{d^{N-2}} \\
     \leq & \frac{C}{|x_1-x_j|^{(N-2)p-2}} + \frac{C(k\varepsilon)^{(N-2)p-N}}{  |x_1-x_j|^{N-2} }.
   \end{split}
\end{equation}

Using Lemma \ref{LemmaA-4}, we have
\begin{equation}\label{en-p9}
    \begin{split}
      \sum\limits_{j=2}^k \frac{|\ln \varepsilon|^{m(p+1)}}{|x_j-x_1|^{(N-2)p-2}} &= O\left( |\ln \varepsilon|^{m(p+1)} \varepsilon^{\frac{(N-2)p-2}{N-2}}  \right) = O\left(\varepsilon^{1+\sigma} \right),\\
      \sum\limits_{j=2}^k \frac{|\ln \varepsilon|^{m(p+1)} (k\varepsilon)^{(N-2)p-N}  }{|x_j-x_1|^{N-2}} &= O\left( |\ln \varepsilon|^{m(p+1)} (k\varepsilon)^{(N-2)p-2} \right)= O\left(\varepsilon^{1+\sigma} \right).
    \end{split}
\end{equation}
Combining \eqref{en-p5}-\eqref{en-p9}, it follows that
$$
I_2=k \sum\limits_{j=2}^k \int_{\Omega_\varepsilon} (PV_{\Lambda,x_1})^{p}PV_{\Lambda,x_j} +
O\left( k \varepsilon^{1+\sigma} \right).
$$
Thus, by Lemma \ref{LE-8} and \eqref{en-important}, we obtain
\begin{equation}\label{en-p10}
  \begin{split}
    I_2 &= k \sum\limits_{j=2}^k \frac{\overline{A}_4 \Lambda^{N-2}}{|x_1-x_j|^{N-2}} +
    O\left( k\sum\limits_{j=2}^k  \frac{1}{|x_1-x_j|^{(N-2)p-2}} + k\sum\limits_{j=2}^k \frac{\varepsilon |\ln \varepsilon|^{m(p+1)}}{|x_1-x_j|^{N-3}}
    + k\varepsilon^{1+\sigma}     \right) \\
    &=k \sum\limits_{j=2}^k \frac{\overline{A}_4 \Lambda^{N-2}}{|x_1-x_j|^{N-2}} + O\left( k\varepsilon^{1+\sigma}     \right).
  \end{split}
\end{equation}

Next, we estimate $I_3$. Note that $|y-x_j| \geq \frac{1}{2} |x_1-x_j|$ for $y\in \Omega_1$, we deduce that
\begin{equation}\label{en-p11}
  \begin{split}
     I_3 =& k \int_{\Omega_1} (PV_{\Lambda,x_1})^{p-1} \left( \sum\limits_{j=2}^k PV_{\Lambda,x_j} \right)^2 \\
     \leq & C k \int_{\Omega_1} \frac{|\ln \varepsilon|^{m(p+1)}}{(1+|y-x_1|)^{(N-2)(p-1)}} \left( \sum\limits_{j=2}^k \frac{1}{(1+|y-x_j|)^{N-2}} \right)^2 \\
     \leq & C k \int_{\Omega_1} \frac{|\ln \varepsilon|^{m(p+1)}}{(1+|y-x_1|)^{(N-2)(p-1)}} \left( \sum\limits_{j=2}^k \frac{1}{(1+|y-x_j|)^{\frac{N-2-\vartheta}{2}}  |x_1-x_j|^{\frac{N-2+\vartheta}{2}}} \right)^2 \\
     \leq & C k \int_{\Omega_1} \frac{|\ln \varepsilon|^{m(p+1)}}{(1+|y-x_1|)^{(N-2)(p-1)}} \left( \sum\limits_{j=2}^k \frac{1}{(1+|y-x_1|)^{\frac{N-2-\vartheta}{2}}  |x_1-x_j|^{\frac{N-2+\vartheta}{2}}} \right)^2 \\
     \leq & C k \int_{\Omega_1} \frac{|\ln \varepsilon|^{m(p+1)}}{(1+|y-x_1|)^{(N-2)p-\vartheta}} \left( \sum\limits_{j=2}^k \frac{1}{ |x_1-x_j|^{\frac{N-2+\vartheta}{2}}} \right)^2 \\
     \leq & C k \int_{\Omega_1} \frac{|\ln \varepsilon|^{m(p+1)}}{(1+|y-x_1|)^{(N-2)p-\vartheta}} \left(   (k\varepsilon)^{\frac{N-2+\vartheta}{2}}   \right)^2 \\
     \leq & C k \varepsilon^{1+\sigma},
   \end{split}
\end{equation}
where $\vartheta > 0$ is a small constant such that $(N-2)p-\vartheta > N$.

Combining \eqref{en-p3}, \eqref{en-p4}, \eqref{en-p10} and \eqref{en-p11}, we obtain
\begin{equation}\label{en-p12}
  \begin{split}
    &\frac{1}{p+1}  \int_{\Omega_\varepsilon} (PV)^{p+1} \\
    =& k \left[ \frac{1}{p+1} \overline{A}_0 + \left( \overline{A}_3 - \frac{1}{p+1} \overline{A}_1   \right)\gamma \Lambda \varepsilon +\sum\limits_{j=2}^k \frac{\overline{A}_4 \Lambda^{N-2}}{|x_1-x_j|^{N-2}}+ O\left(  \varepsilon^{1+\sigma} \right) \right].
  \end{split}
\end{equation}
Then, use a similar argument, we can obtain
\begin{equation}\label{en-p13}
  \begin{split}
    &\frac{1}{q+1}  \int_{\Omega_\varepsilon} (PU)^{q+1} \\
    =& k \left[ \frac{1}{q+1} \overline{B}_0 + \left( \overline{B}_3 - \frac{1}{q+1} \overline{B}_1   \right)\gamma \Lambda \varepsilon +\sum\limits_{j=2}^k \frac{\overline{B}_4 \Lambda^{N-2}}{|x_1-x_j|^{N-2}}+ O\left(  \varepsilon^{1+\sigma} \right) \right].
  \end{split}
\end{equation}

Combining \eqref{en-p1}, \eqref{en-p2}, \eqref{en-p12} and \eqref{en-p13}, it follows that
\begin{equation}\label{en-p14}
\begin{split}
I(PU,PV)=k\left( Q_0 -Q_1 \gamma \Lambda \varepsilon -Q_2 \sum\limits_{j=2}^k \frac{\Lambda^{N-2}}{|x_1-x_j|^{N-2}}  +O(\varepsilon^{1+\sigma}) \right),
\end{split}
\end{equation}
where
\[
  \begin{split}
   Q_0=&\frac{\overline{A}_0+\overline{B}_0}{2} -\frac{\overline{A}_0}{p+1} -\frac{\overline{B}_0}{q+1} > 0,\\
   Q_1=&\frac{\overline{A}_3+\overline{B}_3}{2} +\frac{\overline{A}_1+\overline{B}_1}{2} -\frac{\overline{A}_1}{p+1} -\frac{\overline{B}_1}{q+1} > 0,\\
   Q_2=&\frac{\overline{A}_4 + \overline{B}_4}{2} > 0.
  \end{split}
\]
Similar to the proof of Lemma \ref{LemmaA-4}, there exists a constant $Q_3>0$ such that
$$
\sum\limits_{j=2}^k \frac{1}{|x_1-x_j|^{N-2}} =Q_3(k\varepsilon)^{N-2} + O(\varepsilon^{N-2}k) = Q_3\varepsilon + O(\varepsilon^{N-2}k).
$$
Thus, we have
$$
I(PU,PV)=k\left( Q_0 -Q_1 \gamma \Lambda \varepsilon -Q_4 \Lambda^{N-2} \varepsilon +o(\varepsilon) \right),
$$
where $Q_4=Q_2 Q_3 > 0$. This completes the proof.
\end{proof}

Once we obtain the energy estimate of $(PU,PV)$, we can calculate $F(\Lambda)$. It is desirable for $I(PU,PV)$ to be the main part of $F(\Lambda)$, as this allows us to find the critical point of $F(\Lambda)$ by perturbing the critical point of $I(PU,PV)$ slightly.

\begin{proposition}
   Let $N\geq 5$ and $p$ satisfies the condition (A), then we have
   \[
       F(\Lambda) = k(Q_0 - Q_1\gamma \Lambda \epsilon - Q_4 \Lambda^{N-2}\epsilon + o(\epsilon)),
   \]
   where $Q_i>0, i=0,1,4$ are given in Proposition \ref{prop-energy}.
\end{proposition}

\begin{proof}
   There is a $t \in (0,1)$ such that
   \begin{equation}\label{FLambda}
    \begin{split}
      F(\Lambda) & = I (PU+\omega_1(\Lambda), PV+\omega_2(\Lambda)) \\
      & = I(PU,PV) + \langle I_1^{\prime}(PU,PV) , \omega_1 \rangle + \langle I_2^{\prime}(PU,PV) , \omega_2 \rangle \\
      & + \dfrac{1}{2}(\omega_1,\omega_2) \times
      { \left(
          \begin{array}{cc}
            I_{11}^{\prime\prime}(PU+t\omega_1, PV+t\omega_2) & I_{12}^{\prime\prime}(PU+t\omega_1, PV+t\omega_2) \\
            I_{21}^{\prime\prime}(PU+t\omega_1, PV+t\omega_2) & I_{22}^{\prime\prime}(PU+t\omega_1, PV+t\omega_2)
          \end{array}
      \right)}
      \times
      {\left(
          \begin{array}{c}
            \omega_1 \\
            \omega_2
          \end{array}
      \right). } \\
      & = I(PU,PV) - \int_{\Omega_{\epsilon}} (l_2 \omega_1 + l_1 \omega_2) + \dfrac{1}{2} \int_{\Omega_\epsilon} [(l_1+N_1(\omega_2))\omega_2 + (l_2+N_2(\omega_1))\omega_1 ] \\
      & - \dfrac{p}{2} \int_{\Omega_{\epsilon}} ((PV + t\omega_2)^{p-1} - (PV)^{p-1})\omega_2^2 - \dfrac{q}{2} \int_{\Omega_{\epsilon}} ((PU + t\omega_1)^{q-1} - (PU)^{q-1})\omega_1^2 \\
      & = I(PU,PV) - \dfrac{1}{2} \int_{\Omega_{\epsilon}} (l_2 \omega_1 + l_1 \omega_2) + \dfrac{1}{2} \int_{\Omega_{\epsilon}} (N_1(\omega_2)\omega_2 + N_2(\omega_1)\omega_1) \\
      & - \dfrac{p}{2} \int_{\Omega_{\epsilon}} ((PV + t\omega_2)^{p-1} - (PV)^{p-1})\omega_2^2 - \dfrac{q}{2} \int_{\Omega_{\epsilon}} ((PU + t\omega_1)^{q-1} - (PU)^{q-1})\omega_1^2.
    \end{split}
   \end{equation}
   Note that
   \[
      \int_{\Omega_{\epsilon}} |N_1(\omega_2)||\omega_2| \leq C ||N_1(\omega_2)||_{**,1}||\omega_2||_{*,2} \int_{\Omega_{\epsilon}} \sum\limits_{j=1}^{k} \dfrac{1}{(1+|y-x_j|)^{\frac{N}{q+1}+\tau+2}}\sum\limits_{i=1}^{k} \dfrac{1}{(1+|y-x_j|)^{\frac{N}{p+1}+\tau}},
   \]
   and
   \[
      \begin{split}
         & \quad \sum\limits_{j=1}^{k} \dfrac{1}{(1+|y-x_j|)^{\frac{N}{q+1}+\tau+2}} \sum\limits_{i=1}^{k} \dfrac{1}{(1+|y-x_j|)^{\frac{N}{p+1}+\tau}} \\
         & \leq C \sum\limits_{j=1}^{k} \dfrac{1}{(1+|y-x_j|)^{N+2\tau}} + C \sum\limits_{j=1}^{k} \sum\limits_{i \neq j} \dfrac{1}{(1+|y-x_j|)^{\frac{N}{q+1}+\tau+2}}\dfrac{1}{(1+|y-x_i|)^{\frac{N}{p+1}+\tau}} \\
         & \leq C \sum\limits_{j=1}^{k} \dfrac{1}{(1+|y-x_j|)^{N+2\tau}} + \sum\limits_{j=1}^{k} \dfrac{1}{(1+|y-x_j|)^{N + \frac{\tau}{2}}}  \sum\limits_{i \neq j}\dfrac{C}{|x_i-x_j|^{\frac{3\tau}{2}}} \\
         & \leq C \dfrac{1}{(1+|y-x_j|)^{N+\frac{\tau}{2}}}.
      \end{split}
   \]
   Therefore, we have by Lemma \ref{N} that
   \begin{equation}\label{F-N1}
      \int_{\Omega_{\epsilon}} |N_1(\omega_2)||\omega_2| \leq Ck||N_1(\omega_2)||_{**,1}||\omega_2||_{*,2} \leq Ck||(\omega_1,\omega_2)||_{*}^{min(p+1,3)}.
   \end{equation}
   Similarly, we have
   \begin{equation}\label{F-N2}
      \int_{\Omega_{\epsilon}} |N_2(\omega_1)||\omega_1| \leq Ck||N_2(\omega_1)||_{**,2}||\omega_1||_{*,1} \leq Ck||(\omega_1,\omega_2)||_{*}^{min(p+1,3)}.
   \end{equation}
   On the other hand, by  a similar argument, we obtain
   \begin{equation}\label{F-l}
      \begin{split}
         \int_{\Omega_{\epsilon}} |l_1 \omega_2 + l_2\omega_1| & \leq ||l_1||_{**,1}||\omega_2||_{*,2}  \int_{\Omega_{\epsilon}} \sum\limits_{j=1}^{k} \dfrac{1}{(1+|y-x_j|)^{\frac{N}{q+1}+\tau+2}} \sum\limits_{i=1}^{k} \dfrac{1}{(1+|y-x_j|)^{\frac{N}{p+1}+\tau}}  \\
         & + ||l_2||_{**,2}||\omega_1||_{*,1} \int_{\Omega_{\epsilon}} \sum\limits_{j=1}^{k} \dfrac{1}{(1+|y-x_j|)^{\frac{N}{q+1}+\tau+2}} \sum\limits_{i=1}^{k} \dfrac{1}{(1+|y-x_j|)^{\frac{N}{p+1}+\tau}}  \\
         & \leq C k||(l_1,l_2)||_{**}||(\omega_1,\omega_2)||_{*}.
      \end{split}
   \end{equation}
   Next, we estimat the last two terms on the right side of (\ref{FLambda}). Note that
   \[
      (PV+t\omega_2)^{p-1} - (PV)^{p-1} =
      \begin{cases}
         & O(|\omega_2|^{p-1}),\;\;\hbox{if} \;\; N \geq 6,\\
         & O(|PV|^{p-2} |\omega_2| + |\omega_2|^{p-1}),\;\;\hbox{if} \;\; N = 5.\\
      \end{cases}
   \]
   For $N \geq 6$, we have
   \[
      \begin{split}
         \int_{\Omega_{\epsilon}} |(PV+t\omega_2)^{p-1} - (PV)^{p-1}|\omega_2^2 & \leq  \int_{\Omega_{\epsilon}} |\omega_{2}|^{p+1} dx \\
         & \leq ||\omega_2||_{*,2}^{p+1} \int_{\Omega_{\epsilon}} \left( \sum\limits_{j=1}^k \dfrac{1}{(1+|y-x_j|)^{\frac{N}{p+1}+\tau}} \right)^{p+1} dx.
      \end{split}
   \]
   Without loss of generality, we may assume $y \in \Omega_1$. Then, for arbitrarily small $\eta > 0$, we have
   \[
   \begin{split}
      \sum\limits_{j=2}^k \dfrac{1}{(1+|y-x_j|)^{\frac{N}{p+1}+\tau}} & \leq \sum\limits_{j=2}^k \dfrac{1}{(1+|y-x_1|)^{\tau}}\dfrac{1}{(1+|y-x_j|)^{\frac{N}{p+1}}} \\
      & \leq C \sum\limits_{j=2}^k \dfrac{1}{|x_1-x_j|^{\tau - \eta}} \dfrac{1}{(1+|y-x_1|)^{\frac{N}{p+1} + \eta }} \\
      & \leq C \epsilon^{-\eta} \dfrac{1}{(1+|y-x_1|)^{\frac{N}{p+1} + \eta }}.
   \end{split}
   \]
Thus
   \[
      \begin{split}
      \int_{\Omega_{\epsilon}} \left( \sum\limits_{j=1}^k \dfrac{1}{(1+|y-x_j|)^{\frac{N}{p+1}+\tau}} \right)^{p+1} dx & \leq C k \epsilon^{-\eta (p+1)} \int_{\Omega_1} \dfrac{1}{(1+|y-x_1|)^{N + \eta (p+1) }} \\
      & \leq C k \epsilon^{-\eta (p+1)}.
   \end{split}
   \]
   Hence, we have
   \[
      \int_{\Omega_{\epsilon}} |(PV+t\omega_2)^{p-1} - (PV)^{p-1}|\omega_2^2 \leq Ck\epsilon^{-\eta (p+1)} ||\omega_2||_{*,2}^{p+1}.
   \]
   For $N=5$, by similar estimate, we have
   \[
      \begin{split}
         & \quad \int_{\Omega_{\epsilon}} |(PV+t\omega_2)^{p-1}  - (PV)^{p-1}|\omega_2^2  \leq C \int_{\Omega_{\epsilon}} (|PV|^{p-2}|\omega_2|^3 + |\omega_2|^{p+1}) \\
          & \leq C ||\omega_2||_{*,2}^3 \int_{\Omega_{\epsilon}} \left( \sum\limits_{j=1}^k \dfrac{1}{(1+|y-x_j|)^{N-2}} \right)^{p-2}\left( \sum\limits_{j=1}^k \dfrac{1}{(1+|y-x_j|)^{\frac{N}{p+1}+\tau}} \right)^{3} \\
          & + C |\omega_2||_{*,2}^{p+1} \int_{\Omega_{\epsilon}}\left( \sum\limits_{j=1}^k \dfrac{1}{(1+|y-x_j|)^{\frac{N}{p+1}+\tau}} \right)^{p+1} \\
          & \leq Ck \epsilon^{-\eta (p+1)} ||\omega_{2}||_{*,2}^{3} .
      \end{split}
   \]
   Thus, we have
   \begin{equation}\label{F-error1}
      \int_{\Omega_{\epsilon}} |(PV+t\omega_2)^{p-1} - (PV)^{p-1}|\omega_2^2 \leq Ck\epsilon^{-\eta (p+1)} ||\omega_2||_{*,2}^{min(p+1,3)}.
   \end{equation}
   Similarly,
   \begin{equation}\label{F-error2}
      \int_{\Omega_{\epsilon}} |(PU+t\omega_1)^{p-1} - (PU)^{p-1}|\omega_2^2 \leq Ck\epsilon^{-\eta (q+1)} ||\omega_2||_{*,2}^{min(q+1,3)}.
   \end{equation}
   Combining (\ref{F-N1}), (\ref{F-N2}), (\ref{F-l}), (\ref{F-error1}), (\ref{F-error2}) and considering the fact that $q\geq p$, we can deduce that for sufficiently small $\sigma>0$, it holds that:
   \[
      \begin{split}
          F(\Lambda) & = I(PU,PV) + O(k\epsilon^{-\eta } ||(\omega_1, \omega_2)||_{*}^{min(p+1,3)} + ||(l_1,l_2)||_{**}||(\omega_1,\omega_2)||_{*} ) \\
          & = I(PU,PV) + O(k\epsilon^{1+\sigma}).
      \end{split}
   \]
   Finally, by Proposition \ref{prop-energy}, the proof is completed .
\end{proof}

\section{Proof of Theorem 1.1}

\begin{proof}[Proof of Theorem 1.1]
  Indeed, it is sufficient to prove the existence of a critical point for $F(\Lambda)$ within the interval $[\delta, \delta^{-1}].$ We observe that the function
   \[
      -Q_1 \gamma \Lambda - Q_4\Lambda^{N-2}
   \]
   has a unique maximum at $\Lambda = \left(\frac{-Q_1 \gamma}{Q_4(N-2)}\right)^{\frac{1}{N-3}}$. Consequently, if $\delta > 0$ is sufficiently small, $F(\Lambda)$ will attain a maximum point within the interior of $[\delta, \delta^{-1}]$. Therefore, $F(\Lambda)$ possesses a critical point in the open interval $(\delta, \delta^{-1})$. As a result, $(PU+\omega_1, PV+\omega_2)$ is a solution to problem (\ref{equ-1}) corresponding to such a $\Lambda$.
\end{proof}

We would like to point out that we assume $p$ satisfy condition (A). One of the main reson is that when $p > \frac{N}{N-2}$, the asymptotic behavior of the ground state $U_{0,1}$ to system (\ref{4}) is "good enough". However, in order to make the linear projection $(PU, PV)$ be a "good approximation" (the error term is negligible) of the solution to system (\ref{equ-2}), we need to further restrict the range of $p$. On the other hand, when $p < \frac{N}{N-2}$, the existence of bubbling solution to system (\ref{equ-1}) becomes much more complicated due to the "bad behaviour" of $U_{0,1}$ at infinity. Hence, we need to find a finer approximation instead of $(PU, PV)$ to employ the reduction argument, which will considered in our coming paper.

\appendix

\numberwithin{equation}{section}

\section{Essential Results}

In this section, we will proceed with the estimation of the error terms $\varphi_{\Lambda,x_j}$ and $\psi_{\Lambda,x_j}$, defined as $\varphi_{\Lambda,x_j}(y) = U_{\frac{1}{\Lambda}, x_j}(y) - PU_{\Lambda,x_j}(y)$ and $\psi_{\Lambda,x_j}(y) = V_{\frac{1}{\Lambda}, x_j}(y) - PV_{\Lambda,x_j}(y)$, respectively. These estimations will be based on the methods presented in \cite{N=3, N>=4} and \cite{Neumann-Linli}. However, due to the boundary condition and according to the geometry property of the domain $\Omega.$ We obtain some crucial estimates which  have been extensively utilized in the preceding sections. And we believe these results can be applied to other problems related to the critical system of Hamiltonian type.

\begin{lemma}\label{LemmaA-1}
It holds that
\begin{equation}\label{important-1}
   \varphi_{\Lambda,x_j}(y)=\varepsilon \Lambda^{1-\frac{N}{q+1}} \varphi_0\left(\frac{y-x_j}{\Lambda} \right) + O\left(
   \frac{\varepsilon^{2}|\ln \varepsilon|^m}{  (1+| y-x_j|)^{N-4}   }  +  \varepsilon^{N-2} \right),
\end{equation}
and
\begin{equation}\label{important-2}
\psi_{\Lambda,x_j}(y)=\varepsilon \Lambda^{1-\frac{N}{p+1}} \varphi_0\left(\frac{y-x_j}{\Lambda} \right) + O\left(
   \frac{\varepsilon^{2}|\ln \varepsilon|^m}{  (1+| y-x_j|)^{N-4}   }  +  \varepsilon^{N-2} \right),
\end{equation}
where $\varphi_0$ is defined in \eqref{a-8}, $m=1$ for $N=5$ and $m=0$ for $N \geq 6$.

\end{lemma}

\begin{proof}

Note that
\begin{equation}\label{a-1}
  \begin{cases}
    -\Delta \varphi_{\Lambda,x_j} + \mu \varepsilon^2 \varphi_{\Lambda,x_j} =\mu \varepsilon^2 U_{\frac{1}{\Lambda}, x_j},\;\;\; &\hbox{in } \Omega_\varepsilon,\\
    \frac{\partial \varphi_{\Lambda,x_j}}{\partial n} = \frac{\partial }{\partial n} U_{\frac{1}{\Lambda}, x_j}, &\hbox{on } \partial\Omega_\varepsilon.
  \end{cases}
\end{equation}

Write $\varphi_{\Lambda,x_j}=\varphi_{1}+\varphi_{2}$, where $\varphi_{1}$ is the solution of
\begin{equation}\label{a-2}
  \begin{cases}
    -\Delta \varphi_{1} + \mu \varepsilon^2 \varphi_{1} =\mu \varepsilon^2 U_{\frac{1}{\Lambda}, x_j},\;\;\; &\hbox{in } \Omega_\varepsilon,\\
    \frac{\partial \varphi_{1}}{\partial n} = 0, &\hbox{on } \partial\Omega_\varepsilon.
  \end{cases}
\end{equation}
and $\varphi_{2}$ is the solution of
\begin{equation}\label{a-3}
  \begin{cases}
    -\Delta \varphi_{2} + \mu \varepsilon^2 \varphi_{2} =0,\;\;\; &\hbox{in } \Omega_\varepsilon,\\
    \frac{\partial \varphi_{2}}{\partial n} = \frac{\partial }{\partial n} U_{\frac{1}{\Lambda}, x_j}, &\hbox{on } \partial\Omega_\varepsilon.
   \end{cases}
\end{equation}

Using Lemma \ref{LemmaA-3}, we find that
\begin{equation}\label{a-4}
  \begin{split}
    |\varphi_1(y)| \leq & C \varepsilon^2 \int_{\Omega_\varepsilon} \frac{U_{\frac{1}{\Lambda}, x_j}}{|y-z|^{N-2}} dz \\
    \leq & C \varepsilon^2 \int_{\Omega_\varepsilon} \frac{1}{(1+|z-x_j|)^{N-2}|y-z|^{N-2}} dz \\
    \leq & \frac{C\varepsilon^2|\ln \varepsilon|^m}{(1+|y-x_j|)^{N-4}}.
  \end{split}
\end{equation}

Next we estimate $\varphi_2$. Let $\lambda = \frac{1}{\varepsilon \Lambda}, \widetilde{x}_j = \varepsilon x_j$, and we make a transformation $$\widetilde{\varphi}_2(y) = \varepsilon^{-\frac{N}{q+1}} \varphi_2(\frac{y}{\varepsilon}).$$ Then, $\widetilde{\varphi}_2(y)$ satisfies the following equation:
\begin{equation}\label{a-5}
  \begin{cases}
     -\Delta \widetilde{\varphi}_{2} + \mu \widetilde{\varphi}_{2} =0,\;\;\; &\hbox{in } \Omega,\\
     \frac{\partial \widetilde{\varphi}_{2}}{\partial n} = \frac{\partial }{\partial n} U_{\lambda, \widetilde{x}_j}, &\hbox{on } \partial\Omega.
  \end{cases}
\end{equation}

If $y \notin B_{\delta}(\widetilde{x}_j)$, then $|G(z,y)| \leq C$ for all $z \in B_{\frac{\delta}{2}}(\widetilde{x}_j)$, where $G(z,y)$ is the Green function of $-\Delta + \mu I$ in $\Omega$ with the Neumann boundary condition. Hence, we have
\begin{equation}\label{a-6}
  \begin{split}
    |\widetilde{\varphi}_2(y)| = &\left| \int_\Omega G(z,y)  \frac{\partial }{\partial n} U_{\lambda, \widetilde{x}_j} (z) dz  \right| \\
    \leq & \left|  \int_{\Omega \cap B_{\frac{\delta}{2}}(\widetilde{x}_j)} G(z,y)  \frac{\partial }{\partial n} U_{\lambda, \widetilde{x}_j} (z) dz    \right| +
    \left|  \int_{\Omega \setminus B_{\frac{\delta}{2}}(\widetilde{x}_j)} G(z,y)  \frac{\partial }{\partial n} U_{\lambda, \widetilde{x}_j} (z) dz    \right| \\
   \leq & C \left|  \int_{\Omega \cap B_{\frac{\delta}{2}}(\widetilde{x}_j)} \frac{\partial }{\partial n} U_{\lambda, \widetilde{x}_j} (z) dz    \right|
   + C \left|  \int_{\Omega \setminus B_{\frac{\delta}{2}}(\widetilde{x}_j)} \frac{1}{|y-z|^{N-2}}  \frac{\lambda^{\frac{N}{q+1}}}{\lambda^{N-2} |z-\widetilde{x}_j|^{N-2}   }   \right| \\
   \leq & C \left|  \int_{\Omega \cap B_{\frac{\delta}{2}}(\widetilde{x}_j)} \frac{\lambda^{\frac{N}{q+1}}}{\lambda^{N-2} |z-\widetilde{x}_j|^{N-1}   }   \right|
   +\frac{C}{\lambda^{\frac{N}{p+1}}}   \left|  \int_{\Omega \setminus B_{\frac{\delta}{2}}(\widetilde{x}_j)} \frac{1}{|y-z|^{N-2}}   \right| \\
   \leq & \frac{C}{\lambda^{\frac{N}{p+1}}}  \leq C \varepsilon^{\frac{N}{p+1}},\;\;\;\;\hbox{for } y \notin B_{\delta}(\widetilde{x}_j).
\end{split}
\end{equation}

If $y \in B_{\delta}(\widetilde{x}_j)$, we write
\begin{equation}\label{a-7}
    \widetilde{\varphi}_2(y) = \lambda^{\frac{N}{q+1}-1} \varphi_0(\lambda (y-\widetilde{x}_j)) + \widetilde{\varphi}_3(y) + \widetilde{\varphi}_4(y),
\end{equation}
where $\varphi_0, \widetilde{\varphi}_3, \widetilde{\varphi}_4$ satisfies
\begin{equation}\label{a-8}
   \begin{cases}
     -\Delta \varphi_{0}=0,\;\;\; &\hbox{in } \mathbb R^N_+ = \{(y',y_N) : y_N > 0\},\\
     \frac{\partial \varphi_{0}}{\partial n} = -\frac{N-2}{2} \frac{\sum\limits_{i=1}^N k_i y_i^2}{(1+|y'|^2)^{\frac{N}{2}}}, &\hbox{on } \partial\mathbb R^N_+,
   \end{cases}
\end{equation}

\begin{equation}\label{a-9}
\begin{cases}
-\Delta \widetilde{\varphi}_{3} + \mu \widetilde{\varphi}_{3}=0,\;\;\; &\hbox{in } \Omega,\\
\frac{\partial \widetilde{\varphi}_{3}}{\partial n} = \frac{\partial }{\partial n} U_{\lambda, \widetilde{x}_j} - \frac{\partial }{\partial n}
\left[  \lambda^{\frac{N}{q+1}-1} \varphi_0(\lambda (y-\widetilde{x}_j))       \right], &\hbox{on } \partial\Omega,
\end{cases}
\end{equation}

\begin{equation}\label{a-10}
\begin{cases}
-\Delta \widetilde{\varphi}_{4} + \mu \widetilde{\varphi}_{4}=
(\Delta - \mu)\left[ \lambda^{\frac{N}{q+1}-1} \varphi_0(\lambda (y-\widetilde{x}_j))     \right],\;\;\; &\hbox{in } \Omega,\\
\frac{\partial \widetilde{\varphi}_{4}}{\partial n} = 0, &\hbox{on } \partial\Omega,
\end{cases}
\end{equation}
respectively. Using Green's representation, $\varphi_0$ can be written as
\begin{equation}\label{a-11}
    \varphi_0(y)=\frac{1}{\omega_{N-1}} \sum\limits_{i=1}^{N-1} k_i \int_{\mathbb R^{N-1}} \frac{z_i^2}{(1+|z'|^2)^{\frac{N}{2}}} \frac{1}{|y'-z'|^{N-2}} dz',
\end{equation}
which implies

\begin{equation}\label{a-12}
\begin{split}
|\varphi_0(y)| &\leq \frac{C}{(1+|y|)^{N-3}}, \\
|\nabla \varphi_0(y)| &\leq \frac{C}{(1+|y|)^{N-2}}, \\
|D^2 \varphi_0(y)| &\leq \frac{C}{(1+|y|)^{N-1}}.
\end{split}
\end{equation}

For $\widetilde{\varphi}_4$, we have by using Lemma \ref{LemmaA-3} that
\begin{equation}\label{a-13}
   \begin{split}
     |\widetilde{\varphi}_4| \leq & C \lambda^{\frac{N}{q+1}-1} \int_\Omega \left[ \frac{1}{(1+\lambda |z-\widetilde{x}_j|)^{N-1} |y-z|^{N-2}} +
     \frac{1}{(1+\lambda |z-\widetilde{x}_j|)^{N-3} |y-z|^{N-2}}  \right]dz \\
     \leq & C \varepsilon^{-\frac{N}{q+1}+1} \varepsilon^2 \int_{\Omega_\varepsilon} \frac{1}{(1+|z-x_j|)^{N-3}} \frac{1}{|\varepsilon^{-1} y - z|^{N-2}} dz \\
     \leq & C \varepsilon^{-\frac{N}{q+1}+1} \frac{\varepsilon|\ln \varepsilon|^m}{  (1+|\varepsilon^{-1} y-x_j|)^{N-4}   } .\\
   \end{split}
\end{equation}

For $\widetilde{\varphi}_3$, we have
\begin{equation}\label{a-14}
  \begin{split}
    &\frac{\partial }{\partial n} U_{\lambda, \widetilde{x}_j} = -\frac{N-2}{2} \frac{\lambda^{\frac{N}{q+1}}}{(1+|y'|^2)^{\frac{N}{2}}}
    \left( \sum\limits_{i=1}^{N-1} k_i y_i^2 + O \left(\varepsilon |y'|^3\right)  \right), \\
    &\frac{\partial }{\partial n} \left[  \lambda^{\frac{N}{q+1}-1} \varphi_0(\lambda (y-\widetilde{x}_j))       \right] =
    -\frac{N-2}{2} \frac{\lambda^{\frac{N}{q+1}}}{(1+|y'|^2)^{\frac{N}{2}}} \sum\limits_{i=1}^{N-1} k_i y_i^2 + O \left( \frac{\varepsilon^{1-\frac{N}{q+1}}|y'|}{(1+|y'|)^{N-2}}\right).
  \end{split}
\end{equation}
Then through Green's representation, we also have
\begin{equation}\label{a-15}
|\widetilde{\varphi}_3| \leq C \varepsilon^{-\frac{N}{q+1}+1} \frac{\varepsilon|\ln \varepsilon|^m}{  (1+|\varepsilon^{-1} y-x_j|)^{N-4}   } .
\end{equation}

Combining \eqref{a-6}, \eqref{a-7}, \eqref{a-13} and \eqref{a-15}, we obtain
\begin{equation}\label{a-16}
  \begin{split}
    \widetilde{\varphi}_2(y)=& \lambda^{\frac{N}{q+1}-1} \varphi_0(\lambda (y-\widetilde{x}_j)) + O\left(
    \frac{\varepsilon^{-\frac{N}{q+1}+2}|\ln \varepsilon|^m}{  (1+|\varepsilon^{-1} y-x_j|)^{N-4}   }  +  \varepsilon^{\frac{N}{p+1}} \right) \\
    =& (\Lambda \varepsilon)^{1-\frac{N}{q+1}} \varphi_0\left(\frac{y-\widetilde{x}_j}{\Lambda \varepsilon} \right) + O\left(
   \frac{\varepsilon^{-\frac{N}{q+1}+2}|\ln \varepsilon|^m}{  (1+|\varepsilon^{-1} y-x_j|)^{N-4}   }  +  \varepsilon^{\frac{N}{p+1}} \right).
  \end{split}
\end{equation}

Then from $\widetilde{\varphi}_2(y) = \varepsilon^{-\frac{N}{q+1}} \varphi_2(\frac{y}{\varepsilon})$, we finally obtain
\begin{equation}\label{a-17}
   \varphi_2(y)=\varepsilon \Lambda^{1-\frac{N}{q+1}} \varphi_0\left(\frac{y-x_j}{\Lambda} \right) + O\left(
   \frac{\varepsilon^{2}|\ln \varepsilon|^m}{  (1+| y-x_j|)^{N-4}   }  +  \varepsilon^{N-2} \right).
\end{equation}

Combining \eqref{a-17} and \eqref{a-4}, we obtain \eqref{important-1}. Similarly, we can prove \eqref{important-2}.
\end{proof}

As a direct consequence of Lemma \ref{LemmaA-1}, we have

\begin{lemma}\label{LemmaA-2}
There is a constant $C>0$ such that
$$
|\varphi_{\Lambda,x_j}(y)| \leq \frac{C\varepsilon |\ln \varepsilon|^m}{(1+|y-x_j|)^{N-3}},\;\;|\partial_\Lambda \varphi_{\Lambda,x_j}(y)| \leq \frac{C\varepsilon |\ln \varepsilon|^m}{(1+|y-x_j|)^{N-3}},
$$
and
$$
|\psi_{\Lambda,x_j}(y)| \leq \frac{C\varepsilon |\ln \varepsilon|^m}{(1+|y-x_j|)^{N-3}},\;\;|\partial_\Lambda \psi_{\Lambda,x_j}(y)| \leq \frac{C\varepsilon |\ln \varepsilon|^m}{(1+|y-x_j|)^{N-3}},
$$
where $m=1$ for $N=5$ and $m=0$ for $N \geq 6$.
\end{lemma}

\begin{proof}
   For $y \in \Omega_\varepsilon$, we have $\varepsilon \leq \frac{C}{1+|y-x_j|}$. Then, by using Lemma \ref{LemmaA-1} and \eqref{a-12}, we have
$$
|\varphi_{\Lambda,x_j}(y)| \leq \frac{C\varepsilon|\ln \varepsilon|^m}{(1+|y-x_j|)^{N-3}}.
$$
Differentiating \eqref{a-1} with respect to $\Lambda$, we can repeat the same procedure as in Lemma \ref{LemmaA-1} to obtain
$$
|\partial_\Lambda \varphi_{\Lambda,x_j}(y)| \leq \frac{C\varepsilon|\ln \varepsilon|^m}{(1+|y-x_j|)^{N-3}}.
$$
The estimate to $\psi_{\Lambda,x_j}$ and $\partial_\Lambda \psi_{\Lambda,x_j}(y)$ is similar.

\end{proof}

In order to complete the proof of Lemma \ref{LemmaA-1} and Lemma \ref{LemmaA-2}, we require the following result, which has already been proven in \cite{N>=4}:

\begin{lemma}\label{LemmaA-3}
Let $u$ be the solution of
$$
-\Delta u + \mu \varepsilon^2 u = f \hbox{  in  } \Omega_\varepsilon,\;\;\; \frac{\partial u}{\partial n} = 0 \hbox{  on  } \partial \Omega_\varepsilon.
$$
Then we have
$$
|u(x)| \leq C \int_{\Omega_\varepsilon} \frac{|f(y)|}{|x-y|^{N-2}}dy.
$$
\end{lemma}

Additionally, it is worth noting that $\sum\limits_{j=2}^{k} \frac{1}{|x_1-x_j|^{\alpha}}$ emerges frequently in various estimations. To provide a comprehensive overview of the calculation procedure, we summarize it in the following lemma.

\begin{lemma}\label{LemmaA-4}
It follows that
\[
\begin{split}
&\sum\limits_{j=2}^k \frac{1}{|x_j-x_1|^{\alpha}} = O\left( \varepsilon^\alpha k^\alpha  \right) \;\;\; \hbox{if } \alpha > 1, \\
&\sum\limits_{j=2}^k \frac{1}{|x_j-x_1|^{\alpha}} = O\left( \varepsilon^\alpha k^\alpha \ln k \right) \;\;\; \hbox{if } \alpha = 1, \\
&\sum\limits_{j=2}^k \frac{1}{|x_j-x_1|^{\alpha}} = O\left( \varepsilon^\alpha k  \right) \;\;\; \hbox{if } 0< \alpha < 1.
\end{split}
\]
\end{lemma}

\begin{proof}
Since $x_j=\left( \frac{1}{\varepsilon} \cos\frac{2(j-1)\pi}{k}, \frac{1}{\varepsilon} \sin\frac{2(j-1)\pi}{k} ,0 \right)$, we have
$$
|x_j-x_1| = \frac{2}{\varepsilon} \sin \frac{(j-1)\pi}{k}.
$$

If $k$ is even, we have
$$
\sum\limits_{j=2}^k \frac{1}{|x_j-x_1|^{\alpha}} = 2\left( \frac{\varepsilon}{2}\right)^\alpha \sum\limits_{j=2}^{\frac{k}{2}} \frac{1}{\left( \sin \frac{(j-1)\pi}{k}   \right)^\alpha} + \left(\frac{\varepsilon}{2}\right)^\alpha.
$$

If $k$ is odd, we have
$$
\sum\limits_{j=2}^k \frac{1}{|x_j-x_1|^{\alpha}} = 2\left( \frac{\varepsilon}{2}\right)^\alpha \sum\limits_{j=2}^{\frac{k+1}{2}} \frac{1}{\left( \sin \frac{(j-1)\pi}{k}   \right)^\alpha}.
$$

Note that for $j=2,\cdots,\frac{k+1}{2}$, we have
$$
\frac{2}{\pi} \leq \frac{ \sin \frac{(j-1)\pi}{k}  }{\frac{(j-1)\pi}{k}} \leq 1.
$$

Thus, it follows that
$$
\sum\limits_{j=2}^k \frac{1}{|x_j-x_1|^{\alpha}} = O\left(  \varepsilon^\alpha \sum\limits_{j=2}^{\frac{k+1}{2}} \frac{k^\alpha}{(j-1)^\alpha}    \right)
= O\left( (\varepsilon k)^\alpha \int_1^{\frac{k-1}{2}} \frac{1}{t^\alpha} dt  \right).
$$

Then it follows by direct calculation that
\[
\begin{split}
&\sum\limits_{j=2}^k \frac{1}{|x_j-x_1|^{\alpha}} = O\left( \varepsilon^\alpha k^\alpha  \right) \;\;\; \hbox{if } \alpha > 1, \\
&\sum\limits_{j=2}^k \frac{1}{|x_j-x_1|^{\alpha}} = O\left( \varepsilon^\alpha k^\alpha \ln k \right) \;\;\; \hbox{if } \alpha = 1, \\
&\sum\limits_{j=2}^k \frac{1}{|x_j-x_1|^{\alpha}} = O\left( \varepsilon^\alpha k  \right) \;\;\; \hbox{if } 0< \alpha < 1.
\end{split}
\]

\end{proof}

At the end of this section, we present three crucial lemmas that are cited from the references \cite{ref5}, \cite{ref3} and \cite{Kim}, which play a vital role in our proof.

\begin{lemma} \label{L1}
\cite{ref3} Assume that $p\leq \frac{N+2}{N-2}.$ There exist some positive constants $a=a_{N,p}$ and $b=b_{N,p}$ depending only on $N$ and $p$ such that
$$ \lim\limits_{r\to \infty} r^{N-2} V_{0,1}(r) =b ;$$
while
\begin{equation}\label{7}
\begin{cases}
\lim\limits_{r\to\infty} r^{(N-2)p-2}U_{0,1}(r) =a,  \;\; &\hbox{if } p<\frac{N}{N-2},\\
\lim\limits_{r\to\infty} \frac{r^{N-2}}{\ln r}U_{0,1}(r) =a,  \;\; &\hbox{if } p=\frac{N}{N-2},\\
\lim\limits_{r\to\infty} r^{N-2}U_{0,1}(r) =a,  \;\; &\hbox{if } p>\frac{N}{N-2}.
\end{cases}
\end{equation}
Furthermore, in the last case, we have $b^p=a( (N-2)p-2  )(N-(N-2)p).$
\end{lemma}

\begin{lemma} \label{L2}
   \cite{Kim}Let $(U'_{0,1}(r), V'_{0,1}(r))$ denotes the derivative of $(U_{0,1}(r), V_{0,1}(r))$ with respect to $r$, then there exsits a $C>0$ depending only on $N$ and $p$ such that for $r\geq 1$, the following holds
   \begin{equation}
     |V_{0,1}(r) - \dfrac{b_{N,p}}{r^{N-2}}| \leq \dfrac{C}{r^{N}}, \;\;\;  |V'_{0,1}(r) + \dfrac{(N-2)b_{N,p}}{r^{N-1}}| \leq \dfrac{C}{r^{N+1}}.
   \end{equation}
  Besides, if $ p \in (\frac{N}{N-2}, \frac{N+2}{N-2}] $, then
  \begin{equation}
     |U_{0,1}(r) - \dfrac{a_{N,p}}{r^{N-2}}| \leq \dfrac{C}{r^{N-2+\kappa_{0}}}, \;\;\;  |U'_{0,1}(r) + \dfrac{(N-2)a_{N,p}}{r^{N-1}}| \leq \dfrac{C}{r^{N-1 + \kappa_{0}}}.
  \end{equation}
  where $\kappa_0 = (N-2)p - N > 0$. If $p=\frac{N}{N-2}$, then
  \begin{equation}
     |U_{0,1}(r) - \dfrac{a_{N,p}ln r}{r^{N-2}}| \leq \dfrac{C}{r^{N-2}}, \;\;\;  |U'_{0,1}(r) + \dfrac{(N-2)a_{N,p}ln r}{r^{N-1}}| \leq \dfrac{C}{r^{N-1}}.
  \end{equation}
  If $p \in (\frac{2}{N-2}, \frac{N}{N-2})$, then
  \begin{equation}
     |U_{0,1}(r) - \dfrac{a_{N,p}}{r^{(N-2)p-2}}| \leq \dfrac{C}{r^{(N-2)p-2 + \kappa_1}}, \;\;\;  |U'_{0,1}(r) + \dfrac{((N-2)p-2)a_{N,p}}{r^{(N-2)p-1}}| \leq \dfrac{C}{r^{(N-2)p-1 +\kappa_1}}.
  \end{equation}
  where $\kappa_1 =\in (0,min\{N-(N-2)p, ((N-2)p-2)q_0 - N \})$
  \end{lemma}

\begin{lemma}\label{L3}
   \cite{ref5} Set
   $$(\Psi_{0,1}^0,\Phi_{0,1}^0) = \left(  y \cdot \nabla U_{0,1} +\frac{NU_{0,1}}{q+1},\; y \cdot \nabla V_{0,1}+\frac{NV_{0,1}}{p+1}  \right)$$
   and
   $$ (\Psi_{0,1}^l, \Phi_{0,1}^l) = (\partial_l U_{0,1}, \partial_l V_{0,1} ),\;\; \hbox{for } l=1,\cdots,N. $$
   Then the space of solutions to the linear system
   \begin{equation}\label{8}
   \begin{cases}
   -\Delta \Psi =pV_{0,1}^{p-1} \Phi,\;\;\; \hbox{in } \mathbb R^N,\\
   -\Delta \Phi =qU_{0,1}^{q-1} \Psi,\;\;\; \hbox{in } \mathbb R^N,\\
   (\Psi,\Phi)\in \dot{W}^{2,\frac{p+1}{p}}(\mathbb R^N) \times \dot{W}^{2,\frac{q+1}{q}}(\mathbb R^N),
   \end{cases}
   \end{equation}
   is spanned by
   $$ \left\{  (\Psi_{0,1}^0,\Phi_{0,1}^0), (\Psi_{0,1}^1,\Phi_{0,1}^1)  ,\cdots, (\Psi_{0,1}^N,\Phi_{0,1}^N)    \right\} .$$
   \end{lemma}

\section{Basic Estimate}

In this section, we present some results that are basic in the procedure of Lyapunov-Schmidt reduction. All of these results have been sourced from \cite{YanWei2010} and \cite{Neumann-Linli}.
\begin{lemma}\label{B1}
   For any $\alpha > 0$, we have
   \[
      \sum\limits_{j=1}^{k} \dfrac{1}{(1+|y-x_j|)^{\alpha}} \leq C\left( 1 + \sum\limits_{j=2}^{k}\dfrac{1}{|x_1 - x_j|^{\alpha}} \right).
   \]
   Here, the constant $C>0$ does not depend on $k$.
\end{lemma}

\begin{lemma}\label{B2}
   Suppose $\alpha > 1$ and $\beta > 1$ and $i \neq j$. Then, for any $\sigma \in [0, min (\alpha, \beta)]$, we have
   \[
      \dfrac{1}{(1+|y-x_i|)^{\alpha}}\dfrac{1}{(1+|y-x_j|)^{\beta}} \leq \dfrac{C}{|x_i - x_j|^{\sigma}} \left( \dfrac{1}{(1+|y-x_i|)^{\alpha+\beta-\sigma}} + \dfrac{1}{(1+|y-x_j|)^{\alpha+\beta-\sigma}} \right),
   \]
   where $C$ is positive constant.
\end{lemma}

\begin{lemma}\label{B3}
   If $\sigma \in (0,N-2)$, we have
   \[
      \int_{\R^N} \dfrac{1}{|y-z|^{N-2}} \dfrac{1}{(1+|z|)^{2+\sigma}} dz \leq \dfrac{C}{(1+|y|)^{\sigma}}.
   \]
   If $\sigma > N-2$, we have
   \[
      \int_{\R^N} \dfrac{1}{|y-z|^{N-2}} \dfrac{1}{(1+|z|)^{2+\sigma}} dz \leq \dfrac{C}{(1+|y|)^{N-2}}.
   \]
\end{lemma}
\begin{proof}
   The case for $\sigma \in (0, N-2)$ has already been proven in \cite{YanWei2010}. Therefore, we only need to establish the estimate for $\sigma > N-2$. In particular, we aim to prove the estimate for $|y| \geq 2$. Let $d = |y|$, and we have
   \[
     \begin{split}
       \int_{\R^N} \dfrac{1}{|y-z|^{N-2}} \dfrac{1}{(1+|z|)^{2+\sigma}} dz & = \left( \int_{B_d(0)} +  \int_{B_d(y)} + \int_{\R^N \backslash (B_d(0) \cup B_d(y))}  \right) \dfrac{1}{|y-z|^{N-2}} \dfrac{dz}{(1+|z|)^{2+\sigma}}  \\
     & \leq  \dfrac{1}{d^{N-2}}\int_{B_d(0)}  \dfrac{dz}{(1+|z|)^{2+\sigma}} + \dfrac{1}{|y|^{\sigma+2}} \int_{B_{d}(y)} \dfrac{dz}{|z-y|^{N-2}}  \\
     & +  \dfrac{1}{d^{N-2}} \left( \int_{\R^N \backslash B_d(y)} \dfrac{1}{|y-z|^{\sigma + 2}}  + \int_{\R^N \backslash B_d(0)} \dfrac{1}{(1+|z|)^{\sigma + 2}}  \right) \\
     & \leq \dfrac{C}{|y|^{N-2}} \leq C \dfrac{C}{|y|^{N-2}},
     \end{split}
   \]
   for any $|y| \geq 2$. Hence, we obtain
   \[
      \int_{\R^N} \dfrac{1}{|y-z|^{N-2}} \dfrac{1}{(1+|z|)^{2+\sigma}} dz \leq \dfrac{C}{(1+|y|)^{N-2}}.
   \]
\end{proof}

\begin{lemma}\label{first-term}
   Suppose $(PU, PV)=(\sum\limits_{j=1}^{k} PU_{\Lambda, x_j}, \sum\limits_{j=1}^{k} PV_{\Lambda, x_j})$ and $PU_{\Lambda, x_j}$, $PV_{\Lambda, x_j}$ are defined as the solution of (\ref{equ-3}), then for $\tau = \frac{N-3}{N-2}$ and $N\geq 5$, there is a small $\theta > 0$, such that
   \[
      \begin{split}
        & \int_{\Omega_{\epsilon}} \dfrac{1}{|y-z|^{N-2}} |(PV)^{p-1}(z)\omega_{2}(z)| dz \\
        & \leq C||w_2||_{*,2}\left( \sum\limits_{i=1}^{k} \dfrac{1}{(1+|y-x_i|)^{\frac{N}{q+1}+\tau+ \theta}} + o(1)\sum\limits_{i=1}^{k} \dfrac{1}{(1+|y-x_i|)^{\frac{N}{q+1}+\tau }} \right), \\
        & \int_{\Omega_{\epsilon}} \dfrac{1}{|y-z|^{N-2}} |(PU)^{q-1}(z)\omega_{1}(z)| dz \\
        & \leq C||w_2||_{*,2}\left( \sum\limits_{i=1}^{k} \dfrac{1}{(1+|y-x_i|)^{\frac{N}{p+1}+\tau+ \theta}} + o(1)\sum\limits_{i=1}^{k} \dfrac{1}{(1+|y-x_i|)^{\frac{N}{p+1}+\tau }} \right).
      \end{split}
   \]
\end{lemma}

\begin{proof}
   Since $(PU_{\Lambda,x_i}, PV_{\Lambda,x_{i}})$ satisfies equation (\ref{equ-3}), we obtain by Lemma \ref{LemmaA-3} and Lemma \ref{B3} that
\[
   \begin{split}
 |PV_{\Lambda,x_{i}}|(z) &\leq C \int_{\Omega_{\epsilon}} \dfrac{1}{|z-y|^{N-2}} U_{\frac{1}{\Lambda},x_{i}}^q (y) dy \\
& \leq C \int_{\Omega_{\epsilon}} \dfrac{1}{|z-y|^{N-2}(1+|y-x_i|)^{(N-2)q}} \\
& \leq  \dfrac{C}{|z-x_i|^{N-2}},
   \end{split}
\]
Hence, it follows that
\[
   \begin{split}
      & \quad\int_{\Omega_{\epsilon}} \dfrac{1}{|y-z|^{N-2}} |(PV)^{p-1}(z)\omega_{2}(z)| dz \\
      & \leq ||\omega_{2}||_{*,2} \int_{\Omega_{\epsilon}} \dfrac{1}{|y-z|^{N-2}} |(PV)^{p-1}(z)| \sum\limits_{i=1}^{k}\dfrac{1}{(1+|z-x_i|)^{\frac{N}{p+1}+\tau}} dz \\
      & \leq C ||\omega_{2}||_{*,2} \int_{\R^N} \dfrac{1}{|y-z|^{N-2}} \sum\limits_{i=1}^{k}\dfrac{1}{(1+|z-x_i|)^{(N-2)(p-1)}} \sum\limits_{j=1}^{k}\dfrac{1}{(1+|z-x_j|)^{\frac{N}{p+1}+\tau}} dz\\
      & = C ||\omega_{2}||_{*,2} \sum\limits_{i=1}^{k} \int_{\R^N} \dfrac{1}{|y-z|^{N-2}} \dfrac{1}{(1+|z-x_i|)^{(N-2)(p-1)+\frac{N}{p+1}+\tau}}dz  \\
      & + C ||\omega_{2}||_{*,2} \sum\limits_{i=1}^{k} \sum\limits_{j\neq i} \int_{\R^N} \dfrac{1}{|y-z|^{N-2}} \dfrac{1}{(1+|z-x_i|)^{(N-2)(p-1)}}\dfrac{1}{(1+|z-x_j|)^{\frac{N}{p+1}+\tau}}dz ,
   \end{split}
\]
when $N\geq 6$. Using Lemma \ref{LemmaA-3}, we can deduce for small enough $\theta$ that
\[
   \begin{split}
      & \sum\limits_{i=1}^{k} \int_{\R^N} \dfrac{1}{|y-z|^{N-2}} \dfrac{1}{(1+|z-x_i|)^{(N-2)(p-1)+\frac{N}{p+1}+\tau}} \\
      & \leq \sum\limits_{i=1}^{k} \dfrac{1}{(1+|y-x_i|)^{\frac{N}{q+1}+\tau+ \theta}},
   \end{split}
\]
where the last inequality holds because
$$
min \{ (N-2)(p-1)+\frac{N}{p+1}+\tau -2 , N-2  \} >  \dfrac{N}{q+1} + \tau.
$$
On the other hand, by using Lemma \ref{B2} with $\alpha = (N-2)(p-1)$ and $\beta = \frac{N}{p+1} + \tau$,  we obtain
\[
  \begin{split}
   &\quad \int_{\R^N} \dfrac{1}{|y-z|^{N-2}} \dfrac{1}{(1+|z-x_i|)^{(N-2)(p-1)}}\dfrac{1}{(1+|z-x_j|)^{\frac{N}{p+1}+\tau}}dz \\
   & \leq C\dfrac{1}{|x_i-x_j|^{\tau + \theta}}\int_{\R^N} \dfrac{1}{|y-z|^{N-2}} \left( \dfrac{1}{(1+|z-x_i|)^{\alpha + \beta -\tau - \theta} } + \dfrac{1}{(1+|z-x_j|)^{\alpha + \beta -2 -\tau - \theta}} \right) dz ,
  \end{split}
\]
where $\theta > 0$ is samll enough. Applying Lemma \ref{LemmaA-4}, and considering the fact that $\alpha + \beta - \tau > \frac{N}{q+1} + 2 + \tau$, we conclude that:
\[
  \begin{split}
   & \quad\sum\limits_{i=1}^{k}\sum\limits_{j\neq i}\int_{\R^N} \dfrac{1}{|y-z|^{N-2}} \dfrac{1}{(1+|z-x_i|)^{(N-2)(p-1)}}\dfrac{1}{(1+|z-x_j|)^{\frac{N}{p+1}+\tau}}dz \\
   & \leq C \sum\limits_{i=1}^{k}\sum\limits_{j\neq i} \dfrac{1}{|x_i-x_j|^{\tau + \theta}} \int_{\R^N} \dfrac{1}{|y-z|^{N-2}}  \dfrac{1}{(1+|z-x_i|)^{\frac{N}{q+1} + 2 + \tau }} dz\\
   & \leq C \sum\limits_{i=1}^{k}\sum\limits_{j\neq i} \dfrac{1}{|x_i-x_j|^{\tau + \theta}} \dfrac{1}{(1+|y-x_i|)^{\frac{N}{q+1} + \tau}} \\
   & \leq C \sum\limits_{i=1}^{k} \dfrac{1}{(1+|y-x_i|)^{\frac{N}{q+1}+\tau}} \sum\limits_{j\neq i} \dfrac{1}{|x_i-x_j|^{\tau + \theta}}\\
   & = o(1) \sum\limits_{i=1}^{k} \dfrac{1}{(1+|y-x_i|)^{\frac{N}{q+1}+\tau}}.
  \end{split}
\]
Therefore, we have
\[
\begin{split}
   & \quad\int_{\Omega_{\epsilon}} \dfrac{1}{|y-z|^{N-2}} |(PV)^{p-1}(z)\omega_{2}(z)| dz \\
   & \leq C||w_2||_{*,2}\left( \sum\limits_{i=1}^{k} \dfrac{1}{(1+|y-x_i|)^{\frac{N}{q+1}+\tau+ \theta}} + o(1)\sum\limits_{i=1}^{k} \dfrac{1}{(1+|y-x_i|)^{\frac{N}{q+1}+\tau }} \right).
\end{split}
\]
The case for $N=5$ can be treated in a similar manner, and therefore, we omit it. Besides, similar argument gives the estimate for $PU$. That completes the proof.
\end{proof}

\section*{Acknowledgments}
The authors confirm that there are no relevant financial or non-financial competing interests to report.

\end{document}